\documentclass[12pt, final]{article}

\usepackage[affil-it]{authblk}

\usepackage{hyperref}

\usepackage{amsfonts}
\usepackage[american]{babel}
\usepackage{graphicx}
\usepackage{mathrsfs}
\usepackage{amsmath}
\usepackage{mathrsfs}
\usepackage{accents}
\usepackage{longtable}
\usepackage{color}
\usepackage{booktabs}
\usepackage{newcent}

\newtheorem{proof}{Proof}[section]
\newtheorem{corollary}{Corollary}[proof]
\newtheorem{proposition}{Proposition}[section]

\title{Augmented resolution of linear hyperbolic systems under
  nonconservative form}

\author[1,*]{Adri\'an Navas-Montilla}
\author[2]{Ilhan \"Ozgen-Xian}

\affil[1]{\small Centro Universitario de la Defensa, Universidad de Zaragoza, Spain}
\affil[2]{\small EESA, Lawrence Berkeley National Laboratory, CA, USA}
\affil[*]{\small Corresponding author: anavas@unizar.es}

\date{\small This is a preprint that has not gone through peer
  review.\\ Updated: \today}

\begin{document}

\maketitle

\begin{abstract}
Hyperbolic systems under nonconservative form arise in numerous
applications modeling physical processes, for example from the
relaxation of more general equations (e.g. with dissipative
terms). This paper reviews an existing class of augmented Roe schemes
and discusses their application to linear nonconservative hyperbolic
systems with source terms. We extend existing augmented methods by
redefining them within a common framework which uses a geometric
reinterpretation of source terms. This results in intrinsically
well-balanced numerical discretizations. We discuss two equivalent
formulations: (1) a nonconservative approach and (2) a conservative
reformulation of the problem. The equilibrium properties of the
schemes are examined and the conditions for the preservation of the
well-balanced property are provided. Transient and steady state test
cases for linear acoustics and hyperbolic heat equations are
presented. A complete set of benchmark problems with analytical
solution, including transient and steady situations with
discontinuities in the medium properties, are presented and used to
assess the equilibrium properties of the schemes. It is shown that the
proposed schemes satisfy the expected equilibrium and convergence
properties.
\end{abstract}


\tableofcontents

\section{Introduction}

In this work, our aim is to give a {\it detailed} overview of using
augmented solvers for linear nonconservative hyperbolic systems, with
an emphasis on preserving exact (moving) equilibria. The preservation
of such equilibria is an issue of paramount importance since many
events of interest are just perturbations of an equilibrium
state. However, especially in the presence of source terms, preserving
the correct equilibrium is numerically challenging. Schemes that
recover these solutions exactly are referred to as {\it well-balanced}
schemes \cite{Bermudez_1994}. General frameworks to design
well-balanced numerical schemes for nonlinear nonconservative systems
have been developed in, for example, \cite{Pares_2004,
  LeVeque_2011}. These frameworks are based on the so-called {\it
  path-conservative} methods. In contrast, we will design our
framework based on the augmented solver methods
\cite{Murillo_2016}. These augmented solver-based schemes have been
successfully applied to various hyperbolic conservation and balance
laws, see, for example, \cite{MURILLO20117202}, but have not yet been
discussed in the context of nonconservative systems.

In the following, we will consider linear systems of partial
differential equations in the form of
\begin{equation}\label{eq:systemnDim}
  \frac{\partial U}{\partial t} + \mathcal{A}\left( x \right)
  \frac{\partial U}{\partial x} = S,
\end{equation}
where $U = U \left(x, t\right) \in \Xi \subset \mathbb{R}^{n}$ is the
vector of conserved quantities that takes values on $\Xi$, the set of
admissible states of $U$, $\mathcal{A} = \mathcal{A} \in \mathbb{R}^{n
  \times n}$ is a variable coefficient matrix, and $S$ is the vector
of sources, which will be considered to be of the form $S = S
\left(U\right)$. $n$ denotes the number of equations and unknowns. The
domain of definition for the problem in Eq. (\ref{eq:systemnDim}) is
given by $\Psi= \Omega \times \left[0,T\right]$, hence $x\in \Omega
\subset \mathbb{R}$.

The system in Eq. (\ref{eq:systemnDim}) is said to be {\it
  hyperbolic}, if $\forall x \in \Omega \subset \mathbb{R}$ and
$\forall U \in \Xi$, the matrix $\mathcal{A}$ is diagonalizable with
$n$ real eigenvalues. If the eigenvalues are distinct, then the system
is said to be strictly hyperbolic \cite{Godlewski_1996}.  Further, the
system in Eq. (\ref{eq:systemnDim}) is considered {\it
  nonconservative}, if the matrix $\mathcal{A} \in \mathbb{R}^{n
  \times n}$ is not assumed to be a Jacobian matrix of a conservative
flux \cite{Dumbser_2011b}.  Hyperbolic partial differential equations
often arise from the relaxation of more general problems, such as
those involving dissipative terms. In these cases, the issue of having
infinite wave speeds associated with the diffusive terms is resolved.

The ability of a numerical scheme to preserve the correct equilibrium
is closely related to the {\it equilibrium condition} across a
discontinuity. This condition can be derived by integrating
Eq. (\ref{eq:systemnDim}) inside a sufficiently large spatial
domain. Consider the discontinuous solution of
Eq. (\ref{eq:systemnDim}), located at $x=x_d$ at the initial time
$t=0$, defined as
\begin{equation}\label{eq:weaksoldisc00}
  U\left(x, t \right)=
  \begin{cases}
       U_L & \text{if } x < x_d +\sigma t, \\
       U_R & \text{if } x > x_d +\sigma t,
  \end{cases}
\end{equation}
where the subscripts $L$ and $R$ denote the left and right hand-side
of the discontinuity, and $\sigma$ is the wave celerity. Using an
integration volume $\left[x_L, x_R\right] \in \mathbb{R}$ such that
$x_L<x_d<x_R$, the integration of Eq. (\ref{eq:systemnDim}) inside
this volume yields
\begin{equation}\label{eq:Cond1_0}
  \int_{x_L}^{x_R} \mathcal{A}\left(x\right) \frac{\partial U}
      {\partial x} dx - \sigma \left(U_R - U_L\right) =
      \int_{x_L}^{x_R} S \left(U \right) dx.
\end{equation}
Eq. (\ref{eq:Cond1_0}) is only satisfied if the source term is
considered to be a singular source defined at $x_d$ at the initial
time that travels with the wave---its position in time is given by $
x_d +\sigma t$) \cite{LeVeque_2002}. In what follows, source terms
will be considered to be defined at fixed locations---in other words,
not traveling with the wave. Thus, for traveling waves,
Eq. (\ref{eq:Cond1_0}) reduces to
\begin{equation}\label{eq:Cond1_1}
  \int_{x_L}^{x_R} \mathcal{A}\left(x\right) \frac{\partial U}
      {\partial x} dx - \sigma \left(U_R - U_L\right)  = 0.
\end{equation}
For steady waves with $\sigma=0$, Eq. (\ref{eq:Cond1_0}) reduces to
\begin{equation}\label{eq:Cond1}
  \int_{x_L}^{x_R} \mathcal{A}\left(x\right) \frac{\partial U}
      {\partial x} dx = \int_{x_L}^{x_R} S \left(U \right) dx.
\end{equation}
This separation leads to the generalized Ran\-kine-Hu\-go\-niot
relation (GRH) in the steady state, which are a keystone in the
construction of augmented solvers.

The methods herein used focus on the resolution of non-conservative 
systems with {\it geometric source terms}, defined as a function of 
the system variables and the gradient of a scalar field. In contrast, 
{\it nongeometric source terms} do not include such a gradient in 
their definition. 

In what follows, we will redefine augmented schemes to solve arbitrary
linear hyperbolic systems under nonconservative form. For
well-balancing purposes, we use a {\it geometric reinterpretation of
  source terms}, which is a generalization of the idea introduced in
\cite{bouchut_sommer_zeitlin_2004}. In this method, source terms are
considered in the Riemann problem as singular sources in the sense of
distributions. Our generalization of this method allows to obtain
well-balanced properties even in cases with nongeometrical source
terms. In combination with augmented Riemann solvers, this source term
treatment enables exact equilibrium solutions for linear hyperbolic
systems.

We present two equivalent forms of this numerical scheme: (1) the
fluctuation form and (2) the numerical flux form. The fluctuation form
of our approach can be directly applied to the nonconservative
system. The numerical flux form requires rewriting the system by
augmenting with additional equations accounting for the variation in
time of the system matrix coefficients and additional source terms. In
this way, a conservative flux can be defined and the scheme can be
written in numerical flux form.

We apply these schemes to the linear acoustic equations and to the
hyperbolized heat equation with source term. The latter involves a
stiff relaxation source term which must be exactly balanced in order
to provide a conservative estimation of the heat flux. A complete set
of steady and transient test cases are presented together with their
analytical solution. To emphasize the importance of well-balancing,
the solution provided by the well-balanced schemes is compared with
the solution provided by a wave propagation algorithm that uses a
centered explicit---that is to say not well-balanced---integration of
source terms.

\section{Numerical resolution of nonconservative systems\label{sec:numres-nc}}

Consider the Cauchy problem defined in the domain $\Omega \times
[0,T]$ for the partial differential equation in
Eq. (\ref{eq:systemnDim}) with
\begin{equation}\label{eq:ivbp}
  \left\{
  \begin{aligned}
    &\text{IC: } U\left(x, 0\right) = \accentset{\circ}{U} \left(x\right), \;
    \forall x \in \Omega, \\
    &\text{BC: } U\left(x, t\right) = U_{\partial \Omega} \left(x, t\right) \;
    \forall x \in \partial \Omega,
  \end{aligned}
  \right.
\end{equation}
where IC denotes the initial condition specified by a suitable
function $\accentset{\circ}{U}\left(x\right): \Omega \rightarrow \Xi$,
and BC denotes the boundary condition specified by a suitable function
$U_{\partial \Omega}\left(x, t\right): \Omega \times [0,T] \rightarrow
\Xi$. The spatial domain is given by $\Omega=\left[a, b\right]$, with
$a$ and $b$ being the start and end points of the domain,
respectively. It is discretized in $N$ volume cells, defined as
$\Omega_{i} \subseteq \Omega$, such that $ \Omega \approx
\Omega_{\Delta x} =\bigcup_{i=1}^{N}\Omega_{i}$, with cell edges at
\begin{equation}
  a=x_{\frac{1}{2}}<x_{\frac{3}{2}}<...<x_{N-\frac{1}{2}}<
  x_{N+\frac{1}{2}}=b,
\end{equation}
and with cell size $\Delta x$. In the following, we assume that the
discretized domain $\Omega_{\Delta x}$ is identical to
$\Omega$. Inside each cell, at time $t^n$, the conserved quantities
are defined as cell averages as
\begin{equation}\label{eq:cellaverage1} 
U_i^n=\frac{1}{\Delta x} \int_{x_{i-\frac{1}{2}}}^{x_{i+\frac{1}{2}}} U\left(x,
t^n\right)dx, \quad i=1, \dots, N.
\end{equation}
Integration of the system in Eq. (\ref{eq:ivbp}) inside the control
volume $\Omega_i \times \left[t^n,t^{n+1}\right]$ yields Godunov's
scheme \cite{Godunov_1959}
\begin{equation}\label{eq:system1Dint26}
  U_i^{n+1} = U_i^n - \frac{\Delta t}{\Delta x} \left(
  D_{i+\frac{1}{2}}^- + D_{i-\frac{1}{2}}^+ \right)^n + {\Delta t}
  \bar{S}_{i}^n,
\end{equation}
where $D_{i+\frac{1}{2}}^-$ and $D_{i-\frac{1}{2}}^+$ are the
nonconservative fluctuations, and $\bar{S}_i$ is the numerical
approximation of the integral of the source term as
\begin{equation}\label{eq:centeredSource}
  \bar{S}_{i} \approx \int_{x_{i-\frac{1}{2}}}^{x_{i+\frac{1}{2}}} S\left(U\right) dx.
\end{equation}
$D_{i+\frac{1}{2}}^-$ and $D_{i-\frac{1}{2}}^+$ are computed by
solving the Riemann problem (RP) for Eq. (\ref{eq:systemnDim}) across
the interfaces $i+\frac{1}{2}$ and $i-\frac{1}{2}$ of the cell, respectively. For
example, at the interface $i+\frac{1}{2}$, the RP is
\begin{equation}\label{eq:rpOriginal_f}
  \begin{aligned}
    \accentset{\circ}{U}\left(x\right) =
    \begin{cases}
      U_i    & \text{if } x < x_{i+\frac{1}{2}},\\
      U_{i+1} & \text{if } x > x_{i+\frac{1}{2}}.
    \end{cases}
  \end{aligned}
\end{equation}

Let us study the conditions that such a numerical scheme must satisfy
to preserve equilibrium states. Consider the FV discretization of cell
values in Eq. (\ref{eq:cellaverage1}). In order to preserve an initial
steady state, the equilibrium condition in Eq. (\ref{eq:Cond1}) must
be satisfied at each cell interface. Across a discontinuity at an
interface $i+\frac{1}{2}$, the integration of
Eq. (\ref{eq:systemnDim}) on $\left[ x_i, x_{i+1} \right]$ gives
\begin{equation}\label{eq:cond_discrete}
  \int_{x_i}^{x_{i+1}} \mathcal{A} \left( x \right) \frac{\partial
    U}{\partial x} dx = \int_{x_i}^{x_{i+1}} S \left( U \right) dx.
\end{equation}
The integral of the nonconservative products on the left hand side of Eq. (\ref{eq:cond_discrete}) 
can be approached as
\begin{equation}\label{eq:eqCond120}
\int_{x_i}^{x_{i+1}} \mathcal{A} \left( x \right) \frac{\partial
    U}{\partial x} dx = \left(\widetilde{\mathcal{A}} \delta U \right)_{i+\frac{1}{2}} ,
\end{equation}
where
\begin{equation}\label{eq:eqCond12def}
  \widetilde{\mathcal{A}}_{i+\frac{1}{2}} = \frac{1}{2}\left(\mathcal{A}_i +
  \mathcal{A}_{i+1} \right),
\end{equation}
and $\delta U_{i+\frac{1}{2}} = U_{i+1} - U_{i}$. This allows to
rewrite Eq. (\ref{eq:cond_discrete}) as
\begin{equation}\label{eq:eqCond12}
\left(\widetilde{\mathcal{A}} \delta U \right)_{i+\frac{1}{2}} =
\int_{x_{i}}^{x_{i+1}} S\left(U\right) dx.
\end{equation}

In order to preserve the equilibrium state, suitable numerical
approximations to the integral of the source term that depend on both
the form of the fluctuations and the source term must be derived to
satisfy Eq.  (\ref{eq:eqCond12}). We see that the traditional finite
volume (FV) updating scheme in Eq. (\ref{eq:system1Dint26}) in
combination with Eq. (\ref{eq:centeredSource}) does not intrinsically
preserve equilibria, because we have not specified how the integration
is carried out. Hence, the integral of the source term is not
guaranteed to satisfy Eq. (\ref{eq:eqCond12}) at cell interfaces.

Finding suitable approximations that satisfy Eq. (\ref{eq:eqCond12})
is a subtle issue, commonly studied in depth-averaged geophysical flow
modeling, see for example \cite{Audusse_2004, Bermudez_1994,
  MURILLO20117202, Xia_2018}. In this context, a geometric
reinterpretation of the source term $S\left(U\right)$ yields a
well-balanced numerical scheme with the nature of the source term
\cite{bouchut_sommer_zeitlin_2004}.

\subsection{Geometric reinterpretation of source terms}\label{geometricreintS}

The nonconservative products in Eq. (\ref{eq:eqCond12}) may be exactly
balanced with the source terms using a geometric reinterpretation of
$S\left(U\right)$ as
\begin{equation}\label{eq:Cond12def2}
  S\left(U\right) = \frac{\partial}{\partial x} V\left(x, S\right).
\end{equation}
Here, $V$ is the primitive variable of $S\left(U\right)$ that
satisfies
\begin{equation}\label{eq:geomCont}
  V\left(x, S\right) = \int_{-\infty}^{x} S\left(U\right) dx.
\end{equation}
If the piecewise constant FV discretization of data, given
in Eq. (\ref{eq:cellaverage1}), is also applied to the source term,
$V$ can be evaluated at the cell interfaces as
\begin{equation}\label{eq:geomDisc}
  \left\{
  \begin{aligned}
    &V_{i+\frac{1}{2}} = V\left(0\right) + \sum_{\iota = 1}^i S_\iota \Delta
    x,\\
    &S_i = \frac{1}{\Delta x} \int_{x_{i-\frac{1}{2}}}^{x_{i+\frac{1}{2}}}
    S\left(U\right) dx,
  \end{aligned}
  \right.
\end{equation}
where we set $V\left(0\right) = 0$ for the sake of simplicity. If
considering a linear approximation, cell averages of the primitive
variable $V$ can computed as
\begin{equation}\label{eq:geomDiscAvg}
  V_i = \frac{1}{2} \left(V_{i+\frac{1}{2}} + V_{i-\frac{1}{2}}\right),
\end{equation}
which can also be written as
\begin{equation}\label{eq:geomDiscAvg2}
  V_{i} = V_{i-\frac{1}{2}} + \frac{1}{2} S_i \Delta x.
\end{equation}
Using Eq. (\ref{eq:geomDiscAvg2}), we define
\begin{equation}
    \delta V_{i+\frac{1}{2}} = V_{i+1} - V_i = V_{i+\frac{1}{2}} + \frac{1}{2} S_{i+1}
    \Delta x - V_{i-\frac{1}{2}} - \frac{1}{2} S_i \Delta x.
\end{equation}
 The integral of the source term in Eq. (\ref{eq:eqCond12}), hereafter
referred to as integral of the source term at the interface, can be
approached by using Eq. (\ref{eq:Cond12def2}) as
\begin{equation}\label{eq:intSourceInt}
  \left\{
  \begin{aligned}
    \bar{S}_{i+\frac{1}{2}} &= \int_{x_i}^{x_{i+1}} S\left(U\right) dx,\\
    &= \int_{x_i}^{x_{i+1}} \frac{\partial}{\partial x} V\left(x, S\right) dx,\\
    &\approx \delta V_{i+\frac{1}{2}} \equiv \frac{1}{2} \left(S_i + S_{i+1}\right) \Delta x,
  \end{aligned}
  \right.
\end{equation}
provided a linear approximation. Note that the degree of the
approximation for $V_{i}$ will have an impact on the overall order of
the numerical scheme (e.g. the linear approximation in
Eq. (\ref{eq:geomDiscAvg}) yields a second order of
accuracy). Moreover, when dealing with nonlinear systems and source
terms, a more sophisticated averaging for Eq. (\ref{eq:geomDiscAvg})
might be required.

The steady equilibrium condition that relates the discontinuous left
and right states at the interface $i + \frac{1}{2}$ given by
Eq. (\ref{eq:cond_discrete}) can be obtained by combining
Eq. (\ref{eq:eqCond12}) with Eq. (\ref{eq:intSourceInt}) as
\begin{equation}\label{eq:condbalance}
  \left(\widetilde{\mathcal{A}} \delta U\right)_{i+\frac{1}{2}} =
    \delta V_{i+\frac{1}{2}},
\end{equation}
which states that the jump across cell interfaces under steady state
is
\begin{equation}\label{eq:condbalance2}
  \delta U_{i+\frac{1}{2}} = \frac{1}{2} \widetilde{\mathcal{A}}^{-1}
  \left(S_{i+1} + S_i \right) \Delta x.
\end{equation}

\begin{proposition}\label{prop:exact-schemes}
For a numerical scheme that satisfies the approximate G\-R\-H relation
in Eq.  (\ref{eq:cond_discrete}) at cell interfaces and considers the
linear approximation of the nonconservative product and source term in
Eq.  (\ref{eq:eqCond12}) and Eq. (\ref{eq:intSourceInt}),
respectively, the numerical solution of Eq. (\ref{eq:systemnDim}) will
be equal to the analytical solution (considering machine accuracy) iff
for the analytical solution of Eq.  (\ref{eq:systemnDim}),
$\mathcal{A}\left(U\right) \frac{\partial U}{\partial x}$ can be
expressed as a polynomial of degree at most 1
($\mathcal{A}\left(x\right) \frac{\partial U}{\partial x} \in
P_1(\Omega)$), with $P_1$ the degree 1-polynomial basis and
$x\in\Omega$, and at the same time $S$ can also be expressed as a
polynomial of degree at most 1 ($S \in P_1(\Omega)$).
\end{proposition}

\begin{proof}
  Since the linear approximation of the nonconservative products and
  source term in Eq. (\ref{eq:eqCond12}) and
  Eq. (\ref{eq:intSourceInt}), respectively, provide the exact
  integral if those terms are polynomials of degree at most one, the
  exact GRH relation in Eq. (\ref{eq:cond_discrete}) is recovered at
  every cell interface and the exact solution is thus obtained. Since
  the linear approximation of the nonconservative products and source
  term can not ensure an exact integral if those terms are polynomials
  of degree higher than one, the exact GRH relation in
  Eq. (\ref{eq:cond_discrete}) can not be recovered, and the numerical
  solution deviates from the exact
  solution. Prop. \ref{prop:exact-schemes} is thus proved.
\end{proof}

\begin{corollary}\label{cor:exact-source}
  Let us consider a numerical scheme for the resolution of
  Eq. (\ref{eq:systemnDim}) that satisfies the approximate GRH
  relation in Eq. (\ref{eq:cond_discrete}) at cell interfaces and
  considers the linear approximation of the nonconservative products
  and source term in Eq. (\ref{eq:eqCond12}) and Eq.
  (\ref{eq:intSourceInt}), respectively. Let $q$ be a variable of the
  system in Eq. (\ref{eq:cond_discrete}), that represents a physical
  conservative flux. Further, let $q$ satisfy $\partial q / \partial t
  = 0$ and $\partial q / \partial x=\phi$ at the steady state, where
  $\phi$ (e.g., a source of the quantity associated to $q$), is a
  polynomial of degree at most 1 ( $\phi \in P_1(\Omega)$). Then,
  the numerical scheme will preserve the exact conservation of $q$ at
  the discrete level under steady state. The conditions $\partial q /
  \partial t = 0$ and $\partial q / \partial x = \phi$ will be
  fulfilled with machine accuracy.
\end{corollary}

We will now discuss how to combine this source term treatment with
augmented solvers to obtain a well-balanced numerical scheme.

\subsection{Augmented resolution of the nonhomogeneous system\label{sec:aug-res}}

The augmented Riemann solver approach is based on the consideration of
the source terms directly in the definition of the Riemann problem as
a singular source term. This way, the contribution of the source term
is not explicitly included in the updating scheme for the problem in
Eq. (\ref{eq:ivbp}), but in the definition of the fluctuations,
yielding
\begin{equation}\label{eq:updatingschemeAugmented}
  U_i^{n+1} = U_i^n - \frac{\Delta t}{\Delta x} \left(
  D_{i+\frac{1}{2}}^- + D_{i-\frac{1}{2}}^+ \right)^n.
\end{equation}
The fluctuations must satisfy 
\begin{equation}\label{eq:eqCond1b}
  D_{i+\frac{1}{2}}^- + D_{i+\frac{1}{2}}^+ = \int_{x_i}^{x_{i+1}}
  \mathcal{A}\left(x\right) \frac{\partial U}{\partial x} dx -
  \int_{x_i}^{x_{i+1}} S\left(U\right) dx
\end{equation}
for a suitable approximation of the integrals of the nonconservative
product and the source term. Since we directly average the coefficient matrix, 
and use the geometric reinterpretation of the source term in Eq. 
(\ref{eq:intSourceInt}), Eq. (\ref{eq:eqCond1b}) becomes
\begin{equation}\label{eq:eqCond1c}
  D_{i+\frac{1}{2}}^- + D_{i+\frac{1}{2}}^+ = \left(
  \widetilde{\mathcal{A}} \delta U - \delta V \right)_{i+\frac{1}{2}},
\end{equation}
which recovers the equilibrium condition in Eq. (\ref{eq:condbalance})
at the steady state. In the following, the fluctuations,
$D_{i+\frac{1}{2}}^-$ and $D_{i-\frac{1}{2}}^+$, are derived following
two equivalent approaches. The first approach is based on the
resolution of the nonconservative system, allowing to satisfy
Eq. (\ref{eq:eqCond1c}) by decomposing the source term in the system
matrix eigenvector basis. It is referred to as the {\it fluctuation
  form} of the augmented scheme. The second approach is based on
rewriting the system in conservative form by introducing the
definition of a conservative flux. In this way, the updating scheme
can be written in numerical flux form and Eq. (\ref{eq:eqCond1c}) will
be satisfied by decomposing the source term using the eigenvector
basis of the Jacobian of the conservative flux. Consequently, we refer
to this approach as the {\it flux form} of the augmented scheme.

\subsubsection{Augmented scheme in fluctuation form}\label{sectionfluc}

We consider the RP in Eq. (\ref{eq:rpOriginal_f}) for the hyperbolic
system in Eq. (\ref{eq:systemnDim}), and approximate it by a constant
coefficient linear RP as
\begin{equation}\label{eq:initialsystemforRP3_f}
  \left\{
  \begin{aligned}
    &\text{PDE: } \frac{\partial \hat{U}}{\partial t} +
    \widetilde{\mathcal{A}}_{i+\frac{1}{2}}
    \frac{\partial \hat{U}}{\partial x} = \bar{S}_{i+\frac{1}{2}}\delta_{x=x_{i+1/2}}, \\
    &\text{IC: } \accentset{\circ}{\hat{U}} \left( x \right) =
    \begin{cases}
      U_i    & \text{if } x < x_{i+\frac{1}{2}},\\
      U_{i+1} & \text{if } x > x_{i+\frac{1}{2}},
    \end{cases}
  \end{aligned}
  \right.
\end{equation}
where $\widetilde{\mathcal{A}}_{i+\frac{1}{2}}$ is computed using the 
linear averaging in Eq. (\ref{eq:eqCond12def}),
$\hat{U}\left(x, t\right)$ is the approximate solution of the RP in
Eq. (\ref{eq:rpOriginal_f}) for Eq. (\ref{eq:systemnDim}), and the
singular source term is defined as
\begin{equation}\label{eq:sourcedirac}
   \bar{S}_{i+\frac{1}{2}}\delta_{x=x_{i+1/2}} =
  \begin{cases}
    \delta V _{i+\frac{1}{2}} & \text{if } x = x_{i+\frac{1}{2}}, \\
    0 & \text{else},
  \end{cases}
\end{equation}
allowing to recover the GRH condition in Eq.  (\ref{eq:condbalance})
at the interface. Note that $\delta_{x=x_{i+1/2}}$ is the Dirac mass placed at $x_{i+\frac{1}{2}}$. Since in our case linear averaging of $\mathcal{A}$
yields a Roe-type matrix \cite{Roe_1981},
$\widetilde{\mathcal{A}}_{i+\frac{1}{2}}$ is diagonalizable with
${N_\lambda}=n$ real eigenvalues
$\widetilde{\lambda}^m_{i+\frac{1}{2}}$ and right eigenvectors
$\widetilde{e}^m$ that approximate the eigenvalues and eigenvectors of
$\mathcal{A}$ from the original RP in
Eq. (\ref{eq:rpOriginal_f}). Using
$\widetilde{\lambda}^m_{i+\frac{1}{2}}$ and $\widetilde{e}^m$, we can
construct the matrix $\widetilde{\mathcal{P}}_{i+\frac{1}{2}} = \left[
  \widetilde{e}^1, \widetilde{e}^2, \dots, \widetilde{e}^{N_\lambda}
  \right]$ and its inverse
$\widetilde{\mathcal{P}}^{-1}_{i+\frac{1}{2}}$. Then,
$\widetilde{\mathcal{A}}_{i+\frac{1}{2}}$ can be diagonalized as
\begin{equation}
  \widetilde{\mathcal{A}}_{i+\frac{1}{2}} = \left(
  \widetilde{\mathcal{P}} \widetilde{\Lambda} \widetilde{\mathcal{P}}
  \right)_{i+\frac{1}{2}},
\end{equation}
where $\widetilde{\Lambda}_{i+\frac{1}{2}} =
\text{diag}\left(\widetilde{\lambda}^1, \widetilde{\lambda}^2, \dots,
\widetilde{\lambda}^{N_\lambda}\right)$ is a diagonal matrix with the
approximate eigenvalues as entries.

The system in Eq. (\ref{eq:initialsystemforRP3_f}) can now be
decoupled as
\begin{equation}\label{eq:RPcharacteristicvariables_f}
  \left\{
  \begin{aligned}
    &\text{PDE: } \frac{\partial \hat{W}}{\partial t} +
    \widetilde{\Lambda}_{i+\frac{1}{2}} \frac{\partial
      \hat{W}}{\partial x} = \bar{B}_{i+\frac{1}{2}}\delta_{x=x_{i+1/2}},\\
    &\text{IC: } \accentset{\circ}{\hat{W}} \left(x\right) =
    \begin{cases}
      W_i = \widetilde{\mathcal{P}}^{-1}_{i+\frac{1}{2}} U_i &
      \text{if } x < x_{i+\frac{1}{2}},\\
      W_{i+1} = \widetilde{\mathcal{P}}^{-1}_{i+\frac{1}{2}} U_{i+1} &
      \text{if } x > x_{i+\frac{1}{2}},
    \end{cases}
  \end{aligned}
  \right.
  \end{equation}
 where $\hat{W}= \widetilde{\mathcal{P}}^{-1}_{i+\frac{1}{2}} \hat{U}$
 are the characteristic variables, $\hat{W} =
 \left[\hat{w}^1,\hat{w}^2, \dots, \hat{w}^{N_\lambda}\right]$, and
 $\bar{B}_{i+\frac{1}{2}} =
 \widetilde{\mathcal{P}}^{-1}_{i+\frac{1}{2}} \bar{S}_{i+\frac{1}{2}}$
 is the projection of the source term onto the eigenbasis of the
 coefficient matrix $\widehat{\mathcal{A}}$.

 The general solution $\hat{U}\left(x, t\right)$ can be derived by
 expanding the solution as a linear combination of the eigenvectors as
\begin{equation}\label{lcomb1001_f} 
  \hat{U}\left(x, t\right) = \sum_{m=1}^{N_\lambda} \hat{w}^m \left(x,
  t \right) \widetilde{e}^{m}_{i+\frac{1}{2}},
\end{equation} 
where the scalar values $\hat{w}^m (x,t)$ are the characteristic
solutions and represent the strength of each $m$-wave.

Let $U^-_{i+\frac{1}{2}}$ and $U^+_{i+\frac{1}{2}}$ be the left and
right states in the vicinity of $x_{i+\frac{1}{2}}$, respectively,
defined as
\begin{equation}\label{eq:leftrightstates2_f}
  \left\{
  \begin{aligned}
    U^-_{i+\frac{1}{2}} = \lim_{x \rightarrow x_{i+\frac{1}{2}}^-} \hat{U} \left(x, t\right),\\
    U^+_{i+\frac{1}{2}} = \lim_{x \rightarrow x_{i+\frac{1}{2}}^+} \hat{U} \left(x, t\right),
  \end{aligned}
  \right.
\end{equation}
They can be computed as shown in \cite{Murillo_2016} as
\begin{equation}  \label{eq:lcomb_exppp_f}
  \left\{
  \begin{aligned}
    &U^{-}_{i+\frac{1}{2}}= U_i + \sum_{\lambda^m<0} \left[ \left(
      \alpha - \frac{\bar{\beta}}{\widetilde{\lambda}}\right)
      \widetilde{e}\right]^{m}_{i+{\frac{1}{2}}},
    \\&U^{+}_{i+\frac{1}{2}}= U_{i+1} - \sum_{\lambda^m>0} \left[
      \left( \alpha - \frac{\bar{\beta}}{\widetilde{\lambda}}\right)
      \widetilde{e}\right]^{m}_{i+{\frac{1}{2}}},
  \end{aligned}
  \right.
\end{equation} 
where the set of wave and source strengths, $\vec{\alpha}$ and
$\vec{\beta}$, respectively, are defined as
\begin{equation}
  \left\{
  \begin{aligned}
    &\vec{\alpha}_{i+\frac{1}{2}}= \left[ \alpha ^1, \alpha^2,
      \dots, \alpha^{N_\lambda} \right]^T_{i+{\frac{1}{2}}} =
    \left(\widetilde{\mathcal{P}}^{-1} \delta U\right)_{i +
      \frac{1}{2}},\\
    &\vec{\beta}_{i+\frac{1}{2}}=\left[\beta^1, \beta^2, \dots,
    \beta^{N_\lambda} \right]_{i+\frac{1}{2}}^T =
    \left(\widetilde{\mathcal{P}}^{-1} \delta V
    \right)_{i+\frac{1}{2}}.
  \end{aligned}
  \right.
\end{equation} 
The fluctuations are calculated as
\begin{equation}\label{eq:flucts}
\left\{
\begin{aligned}
  D_{i+\frac{1}{2}}^-= \sum_{\lambda^m<0} \left[
    \left(\widetilde{\lambda} \alpha - \beta\right)
    \widetilde{e}\right]^{m}_{i+\frac{1}{2}},
  \\ D_{i+\frac{1}{2}}^+= \sum_{\lambda^m>0} \left[ \left(
    \widetilde{\lambda}\alpha - \beta\right)
    \widetilde{e}\right]^{m}_{i+\frac{1}{2}}.
\end{aligned}
\right.
\end{equation} 
This satisfies
\begin{equation}  \label{eq:relationnsource1} 
  \delta V_{i+\frac{1}{2}} = \widetilde{\mathcal{A}} _{i+\frac{1}{2}}
  \left( U^+_{i+\frac{1}{2}} - U^-_{i+\frac{1}{2}} \right),
\end{equation} 
which under steady state yields 
\begin{equation}  \label{eq:relationnsource2} 
  \delta V_{i+\frac{1}{2}} = \widetilde{\mathcal{A}} _{i+\frac{1}{2}}
  \delta U_{i+\frac{1}{2}} ,
\end{equation} 
allowing to recover Eq. (\ref{eq:condbalance}).

\subsubsection{Augmented scheme in  flux form}\label{sectionflux}

An equivalent to the numerical scheme in Sec. \ref{sectionfluc} can be
derived by augmenting the system to include the elements of
$\mathcal{A}$ in Eq. (\ref{eq:systemnDim}) as conserved quantities,
hereafter denoted by $a_{i,j}$. This is accomplished by adding the
trivial equation
\begin{equation}\label{eq:extraEq}
  \frac{\partial a_{i,j}}{\partial t} = 0
\end{equation}
to the system of equations.

This gives augmented vectors of the size $n + n^2$, and we will denote
these augmented vectors by a bar symbol (e.g., $\bar{U}$ for the
augmented vector of conserved variables). These vectors are written as
\begin{equation}\label{eq:aug1}
  \bar{U} =\left[
    \begin{array}{c|c}
      U & a_{1,1}, \dots, a_{1,n}, a_{2,1}, \dots, a_{2,n}, \dots
      ,a_{n,1}, \dots ,a_{n,n}
    \end{array}\right]^T,
\end{equation}
$\bar{U} \in \Xi \subset \mathbb{R}^{n + n^2}$. The augmented
vector of sources $\bar{S}$ is constructed as
\begin{equation}\label{eq:aug2}
  \bar{S}=\left[
    \begin{array}{c|c}
      S & 0, \dots, 0
    \end{array}\right]^T ,
\end{equation}
and represents the following mapping $\bar{S}: \Xi \rightarrow
\mathbb{R}^{n+n^2}$. The augmented version of the system matrix
$\mathcal{A}$ is defined as the block-structured matrix
\begin{equation}\label{eq:aug5}
  \bar{\mathcal{A}} = \left[
  \begin{array}{c|c}
    \mathcal{A} & 0\\ \hline
    0 & 0
  \end{array}\right], \; \bar{\mathcal{A}} \in \mathbb{R}^{(n+n^2) \times (n+n^2)}.
\end{equation}
Using Eq. (\ref{eq:aug1}), Eq. (\ref{eq:aug2}) and
Eq. (\ref{eq:aug5})), the augmented form of the system in
Eq. (\ref{eq:systemnDim}) is written as
\begin{equation}\label{eq:systemnDimAug}
\frac{\partial \bar{U}}{\partial t} +
\bar{\mathcal{A}}\left(x\right) \frac{\partial \bar{U}}{\partial
  x} = \bar{S}.
\end{equation}
Linear nonconservative systems can be cast in conservative form by
applying the chain rule of calculus and adding an additional
noncon\-ser\-va\-tive geometric source term that allows to recover
the original system. In this way, Eq. (\ref{eq:systemnDim}) is
rewritten in an extended form as
\begin{equation}\label{eq:systemAugmented}
  \frac{\partial \bar{U}}{\partial t} + \frac{\partial
    \bar{F}}{\partial x} = \bar{\mathcal{K}} \frac{\partial
    \bar{U}}{\partial x} + \bar{S}.
\end{equation}
where $\bar{F}$ is the vector of conservative fluxes, given by a
linear mapping $\bar{F}: \Xi \rightarrow \mathbb{R}^{n+n^2}$ and
defined as
\begin{equation}\label{eq:defFcons}
  \bar{F} = \bar{\mathcal{M}} \bar{U},
\end{equation}
where $\bar{\mathcal{M}} \in \mathbb{R}^{\left(n + n^2\right) \times
  \left(n + n^2 \right)}$ is the Jacobian matrix of the conservative
flux, computed as
\begin{equation}\label{eq:aug6}
  \bar{\mathcal{M}} = \bar{\mathcal{A}} + \bar{\mathcal{K}} = \left[
	\begin{array}{c|c}
	\mathcal{A} & \mathcal{B}\\ \hline
	0 & 0
	\end{array}
        \right] \in \mathbb{R}^{(n+n^2) \times (n+n^2)}.
\end{equation}
Here, $\bar{\mathcal{K}} \in \mathbb{R}^{\left(n+n^2\right) \times
  \left(n+n^2\right)}$ is the coefficient matrix associated with the
noncon\-ser\-va\-tive terms, written as
\begin{equation}\label{eq:aug3}
  \bar{\mathcal{K}}=\left[
  \begin{array}{c|c}
    0 & \mathcal{B}\\
    \hline
    0 & 0
  \end{array}\right],
\end{equation}
and $\mathcal{B} \in \mathbb{R}^{n^2 \times n}$ is the following
diagonal block-structured matrix
\begin{equation}\label{eq:aug4}
  \mathcal{B} = \text{diag} \left( U^T \right).
\end{equation}

\begin{proposition}\label{prop:prop1}
  The matrix $\bar{\mathcal{M}}$ can be diagonalized as $\bar{\Lambda}
  = \bar{\mathcal{P}}^{-1} \bar{\mathcal{M}} \bar{\mathcal{P}}$, with
  $\bar{\mathcal{P}}$ being the column matrix of the right
  eigenvectors of $\bar{\mathcal{M}}$, yielding
  \begin{equation}\label{eq:aug5-1}
    \bar{\Lambda} = \left[
    \begin{array}{c|c}
      \Lambda & 0\\
      \hline
      0 & 0
    \end{array}
    \right], \; \bar{\Lambda} \in \mathbb{R}^{(n+n^2) \times (n+n^2)},
  \end{equation}
  where $\Lambda \in \mathbb{R}^{n\times n} $ is $\Lambda =
  \mathcal{P}^{-1} \mathcal{A} \mathcal{P}$, with $\mathcal{P}$ being
  the column matrix of the right eigenvectors of $\mathcal{A}$.
\end{proposition}
\begin{proof}
Using the Schur complement (see, e.g., \cite{Haynsworth_1968}), it can
be shown that the eigenvalues of a block matrix are the combined
eigenvalues of its blocks (omitted here). Since all eigenvalues of
$\mathcal{B}$ equal $0$, the eigenvalues of $\bar{\mathcal{M}}$ equal
the eigenvalues of $\mathcal{A}$ augmented by $n^2$ zeros. This
corresponds to $\bar{\Lambda}$ given in Eq.  (\ref{eq:aug5-1}), and
Prop. (\ref{prop:prop1}) is therefore proved.
\end{proof}

\begin{corollary}\label{eq:col1}
If the homogeneous part of the system in Eq. (\ref{eq:systemnDim}) is
hyperbolic, then the homogeneous part of the system in
Eq. (\ref{eq:systemAugmented}) is also hyperbolic.
\end{corollary}

\begin{corollary}\label{col2}
The eigenvalues of $\bar{\mathcal{M}}$, denoted as $\bar{\lambda}^m$,
are given by the diagonal elements of $\bar{\Lambda}$. They are equal
to
\begin{equation}
  \bar{\lambda}^m =
  \begin{cases}
        {\lambda}^m & \text{if } 1\leq m\leq n ,\\
        0 & \text{if } n < m \leq n+n^2,
  \end{cases}
\end{equation}
 where ${\lambda}^m$ are the eigenvalues of $\mathcal{A}$.
\end{corollary}

We construct a numerical scheme in the form of
Eq. (\ref{eq:updatingschemeAugmented}), using the augmented vectors
and matrices. The fluctuations can be written in terms of numerical
fluxes as
\begin{equation}
  \left\{
  \begin{aligned}
    &\bar{D}^-_{i+\frac{1}{2}} = \bar{F}^-_{i+\frac{1}{2}} - \bar{F}_i,\\
    &\bar{D}^+_{i+\frac{1}{2}} = \bar{F}_i - \bar{F}^+_{i-\frac{1}{2}},
  \end{aligned}
  \right.
\end{equation}
with $\bar{F}^-_{i+\frac{1}{2}}$ and $\bar{F}^+_{i-\frac{1}{2}}$ being
the numerical fluxes at cell interfaces. Then,
Eq. (\ref{eq:updatingschemeAugmented}) becomes
\begin{equation}\label{eq:updatingschemeAugmented2}
  \bar{U}_i^{n+1} = \bar{U}_i^n - \frac{\Delta t}{\Delta x} \left(
  \bar{F}_{i+\frac{1}{2}}^- - \bar{F}_{i-\frac{1}{2}}^+ \right)
\end{equation}
A numerical scheme in the form of
Eq. (\ref{eq:updatingschemeAugmented2}) solves the RP
\begin{equation}\label{eq:rpfluxaug}
  \left\{
  \begin{aligned}
    &\text{PDE: } \frac{\partial \bar{U}}{\partial t} + \frac{\partial
      \bar{F} \left(\bar{U}\right)}{\partial x} =
    \bar{\mathcal{K}}\left( \bar{U}\right) \frac{\partial
      \bar{U}}{\partial x} + \bar{S} \left( \bar{U}\right),\\
    &\text{IC: } \accentset{\circ}{\bar{U}} \left(x\right) =
    \begin{cases}
      \bar{U}_i & \text{if } x < x_{i+\frac{1}{2}},\\
      \bar{U}_{i+1} & \text{if } x > x_{i+\frac{1}{2}}.
    \end{cases}
  \end{aligned}
  \right.
\end{equation}
at the interface $i+\frac{1}{2}$. As in the previous section, we
approximate Eq. (\ref{eq:rpfluxaug}) with a constant coefficient
RP as
\begin{equation}\label{eq:rpAppr}
  \left\{
  \begin{aligned}
    &\text{PDE: } \frac{\partial \bar{U}}{\partial t} +
    \widetilde{\bar{J}}_{i+\frac{1}{2}} \frac{\partial
      \bar{U}}{\partial x} =
    \bar{\bar{S}}_{i+\frac{1}{2}}\delta_{x=x_{i+1/2}},\\
    &\text{IC: }
    \accentset{\circ}{\bar{U}} \left(x\right) =
    \begin{cases}
      \bar{U}_i & \text{if } x < x_{i+\frac{1}{2}},\\
      \bar{U}_{i+1} & \text{if } x > x_{i+\frac{1}{2}}.
    \end{cases}
  \end{aligned}
  \right.
\end{equation}
where $\widetilde{\bar{J}}_{i+\frac{1}{2}}$ is the approximate
Jacobian of the conservative flux $\bar{F}\left(\bar{U}\right)$, which
satisfies
\begin{equation}\label{eq:roe requirement}
  \delta \bar{F}_{i+\frac{1}{2}} = \widetilde{\bar{J}}_{i+\frac{1}{2}}
  \delta \bar{U}_{i+\frac{1}{2}},
\end{equation}
and is defined using Roe's averaging procedure \cite{Roe_1981}
\begin{equation}\label{eq:roematrix}
  \widetilde{\bar{J}}_{i+\frac{1}{2}} \equiv \widetilde{
    \bar{\mathcal{M}}}_{i+\frac{1}{2}} = \frac{1}{2}
  \left(\bar{\mathcal{M}}_i + \bar{\mathcal{M}}_{i+1} \right).
\end{equation}
Let $\widetilde{\bar{\mathcal{K}}}_{i+\frac{1}{2}}$ be an approximate
matrix of $\mathcal{K}$ at the interface $i + \frac{1}{2}$ that
satisfies the equality
\begin{equation}\label{eq:aprKmatrix1}
	\int_{\bar{U}_i}^{\bar{U}_{i+1}}
        \bar{\mathcal{K}}\left(\bar{U}\right) d\bar{U} = \left(
        \widetilde{\bar{\mathcal{K}}} \delta \bar{U}
        \right)_{i+\frac{1}{2}},
\end{equation}
provided a linear averaging. The term $\bar{\bar{S}}_{i+\frac{1}{2}}$ is
defined as a singular source term as follows
\begin{equation}\label{eq:sourcedirac2}
  \bar{\bar{S}}_{i+\frac{1}{2}}\delta_{x=x_{i+1/2}}=
  \begin{cases}
    \left(\widetilde{\bar{\mathcal{K}}} \delta \bar{U} + \delta
    \bar{V}\right)_{i + \frac{1}{2}} & \text{if } x = x_{i+\frac{1}{2}},\\
    0 & \text{else}.
  \end{cases}
\end{equation}
Since in our case Eq. (\ref{eq:roematrix}) yields a Roe matrix
\cite{Roe_1981}, $\widetilde{\bar{J}}_{i+\frac{1}{2}}$ is
diagonalizable with ${N_\lambda}=n+n^2$ real eigenvalues
$\widetilde{\lambda}^m_{i+{\frac{1}{2}}}$, and right eigenvectors
$\widetilde{e}^m$. Consequently, $\widetilde{\bar{\mathcal{J}}}$ can
be diagonalized using the column matrix of right eigenvectors
$\widetilde{\bar{\mathcal{P}}}_{i+\frac{1}{2}}$ and its inverse
$\widetilde{\bar{\mathcal{P}}}_{i+\frac{1}{2}}^{-1}$ as
$\widetilde{\bar{J}}_{i+\frac{1}{2}} =
\left(\widetilde{\bar{\mathcal{P}}} \widetilde{\Lambda}
\widetilde{\bar{\mathcal{P}}}^{-1}\right)_{i+\frac{1}{2}}$ with
$\widetilde{\Lambda} _{i+\frac{1}{2}} = \text{diag}\left(
\widetilde{\lambda}^{
  1},\dots,\widetilde{\lambda}^{N_\lambda}\right)$.

The system in Eq. (\ref{eq:rpAppr}) can now be decoupled using
$\widetilde{\bar{\mathcal{P}}}^{-1}$, as previously done in Eq.
(\ref{eq:RPcharacteristicvariables_f}). Analogously, an approximate
flux function $\hat{\bar{F}}\left(x, t\right)$ with a similar
structure as $\hat{\bar{U}}\left(x, t\right)$ can also be
constructed. Intercell values for the fluxes in the vicinity of
interface $i+\frac{1}{2}$ are defined as
\begin{equation} \label{eq:leftrightstates3} 
  \left\{
  \begin{aligned}
    &\bar{F}^{-}_{i+\frac{1}{2}}=\lim_{x\rightarrow x_{i+\frac{1}{2}}^-}  \hat{\bar{F}}\left(x, t\right),\\
    &\bar{F}^{+}_{i+\frac{1}{2}}=\lim_{x\rightarrow x_{i+\frac{1}{2}}^+}  \hat{\bar{F}}\left(x, t\right).
  \end{aligned}
  \right.
\end{equation}
Here, the approximate fluxes on the left and right side of
$x_{i+\frac{1}{2}}$, $\bar{F}^{-}_{i}$ and $\bar{F}^{+}_{i+1}$,
respectively, read
\begin{equation}\label{eq:slinearsolutionF}
  \left\{
  \begin{aligned}
    &\bar{F}^-_{i+\frac{1}{2}} = \bar{F}_{i} + \sum_{\lambda^m<0}
    \left[\left(\widetilde{\lambda} \alpha - \bar{\beta}\right)
      \widetilde{e}\right]^{m}_{i+\frac{1}{2}},\\
    &\bar{F}^+_{i+\frac{1}{2}} = \bar{F}_{i+1} - \sum_ {\lambda^m>0}
    \left[\left(\widetilde{\lambda} \alpha - \bar{\beta}\right)
      \widetilde{e}\right]^{m}_{i+\frac{1}{2}},
  \end{aligned}
  \right.
\end{equation} 
allowing to define the fluctuations as
\begin{equation}  \label{eq:flucts2}
  \left\{
  \begin{aligned}
    &D_{i+\frac{1}{2}}^- = \sum_{\lambda^m<0}
    \left[\left(\widetilde{\lambda} \alpha - \bar{\beta} \right)
      \widetilde{e}\right]^{m}_{i+\frac{1}{2}},\\
    &D_{i+\frac{1}{2}}^+ = \sum_{\lambda^m>0} \left[\left(
      \widetilde{\lambda} \alpha - {\bar{\beta}}\right) \widetilde{e}
      \right]^{m}_{i+\frac{1}{2}},
  \end{aligned}
  \right.
\end{equation} 
where the set of wave and source strengths, $\bar{A}_{i+\frac{1}{2}}$
and $\bar{B}_{i+\frac{1}{2}}$, respectively, are defined as
\begin{equation}
  \left\{
  \begin{aligned}
    &\bar{A}_{i+\frac{1}{2}} = \left[ \alpha^1, \alpha^2, \dots,
      \alpha^{N_\lambda} \right]^T_{i+\frac{1}{2}}
    = \left(\widetilde{\bar{\mathcal{P}}}^{-1} \delta \bar{U}\right)_{i+\frac{1}{2}},\\
    &\bar{B}_{i+\frac{1}{2}} = \left[ \beta^1, \beta^2, \dots,
      \beta^{N_\lambda} \right]^T_{i+\frac{1}{2}}
    = \left(\widetilde{\bar{\mathcal{P}}}^{-1} \left(
    \widetilde{\bar{\mathcal{K}}} \delta \bar{U} + \delta \bar{V}
    \right) \right)_{i+\frac{1}{2}}.
  \end{aligned}
  \right.
\end{equation} 
This scheme satisfies
\begin{equation}\label{eq:relationnsource1_fluxes}
  \left( \widetilde{\bar{\mathcal{K}}} \delta \bar{U} + \delta \bar{V}
  \right)_{i+\frac{1}{2}} = \bar{F}^+_{i+\frac{1}{2}} -
  \bar{F}^-_{i + \frac{1}{2}},
\end{equation} 
which under steady state yields
\begin{equation}\label{eq:finaleq}
  \left( \widetilde{\bar{\mathcal{K}}} \delta \bar{U} + \delta \bar{V}
  \right)_{i+\frac{1}{2}} = \left( \widetilde{\bar{\mathcal{M}}} \delta
  \bar{U} \right)_{i + \frac{1}{2}}.
\end{equation} 
The relation in Eq. (\ref{eq:aug6}) allows to recover Eq.
(\ref{eq:condbalance}) from Eq. (\ref{eq:finaleq}), and hence the
augmented scheme in fluctuation form (Sec. \ref{sectionfluc}) and flux
form (Sec. \ref{sectionflux}) are equivalent.

These rather long derivations might be perceived as unnecessarily 
complicated in the context of linear systems, but may become useful
when extending these approaches to the resolution of nonlinear systems.

\section{Computational test cases\label{sec:comp-test}}

We study the performance of the presented approaches using two linear
hyperbolic systems under nonconservative form. We present transient
simulations for the linear acoustic equations, and both transient and
steady state simulations for the hyperbolized heat equation. For the
augmented solver approach, the fluctuation form and the flux form are
equivalent and yield the same results up to machine accuracy. Below,
we only show the results obtained by using the fluctuation form.

\subsection{Linear acoustics}

The linear acoustic equations are obtained by linearizing the 
isentropic Euler equations \cite{LeVeque_2002}. In the one-dimensional 
case (1D), the system reads
\begin{equation}\label{eq:acousticeq1D}
  \left\{
  \begin{aligned}
    \frac{\partial p}{\partial t} + K(x) \frac{\partial u}{\partial x} =0 \\ 
    \frac{\partial u}{\partial t}  + \frac{1}{\rho(x)} \frac{\partial p}{\partial x}  = 0,
  \end{aligned}
  \right.  
\end{equation}
where $p$ is the pressure, $u$ is the velocity, $K$ is the bulk
modulus of elasticity and $\rho$ is the density. The linear acoustic
equations describe the propagation of small amplitude perturbations of
$p$ and $u$ in the medium. See Appendix \ref{app:acoustic} for the eigenvalues
and eigenvectors of the system.

\subsubsection{Transient solution considering a piecewise constant density}

The augmented scheme is applied here to the 1D linear acoustic
equations to simulate an extreme case from \cite{LeVeque_1997}, where
the speed of sound is discontinuous at an interface between two
media. The performance of the augmented scheme is assessed by
comparing to a high-resolution reference solution computed by the wave
propagation algorithm proposed in \cite{LeVeque_1997,LeVeque_2002}.

The spatial domain is given by $\Omega=[0,1]$ and the simulation time
is $T=0.52$. The properties of the media are
\begin{equation}\label{eq:mediadef1}
\left\{
\begin{aligned}
&K\left(x\right) = 1, \\
&\rho\left(x\right) =
\begin{cases}
    1 & \text{if } x < 0.6, \\
    4 & \text{if } x > 0.6,
\end{cases}
\end{aligned}
\right.
\end{equation}
giving a jump in the wave velocity from $c=1$ on the left region to
$c=0.5$ on the right region. As initial condition, we impose $u=0$ and
a hump in pressure given by
\begin{equation}
 p\left(x\right) =
\begin{cases}
    \hat{p} \sqrt{1 - \left(\frac{x - x_0}{\hat{x}}\right)^2} 
    & \text{if } \left| x - x_0 \right| < \hat{x}, \\
    0 & \text{else},
\end{cases}
\end{equation}
where $x_0 = 0.4$, $\hat{x} = 0.075$ and $\hat{p} = 0.2$. 

We run computations on different mesh sizes from $\Delta x = 0.01$
down to $\Delta x = 0.0025$ using a Courant, Friedrichs and Lewy
condition (CFL) \cite{Courant_1928} with a CFL number of $0.8$. The
solution converges to the reference solution as the grid is
refined. This is seen in Fig. \ref{fig:res_ac1}, where results are
plotted at $t=0.104$, $t=0.26$, $t=0.364$ and $t=0.52$.

\begin{figure}
    \begin{center}
    \includegraphics[width=0.49\textwidth]{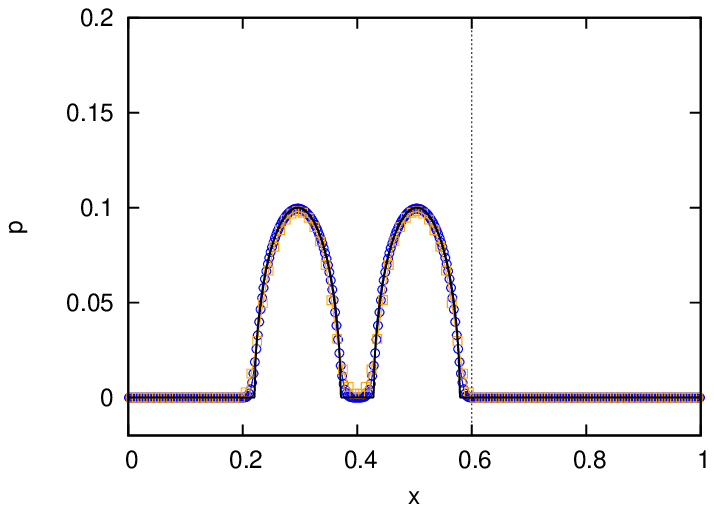}
    \includegraphics[width=0.49\textwidth]{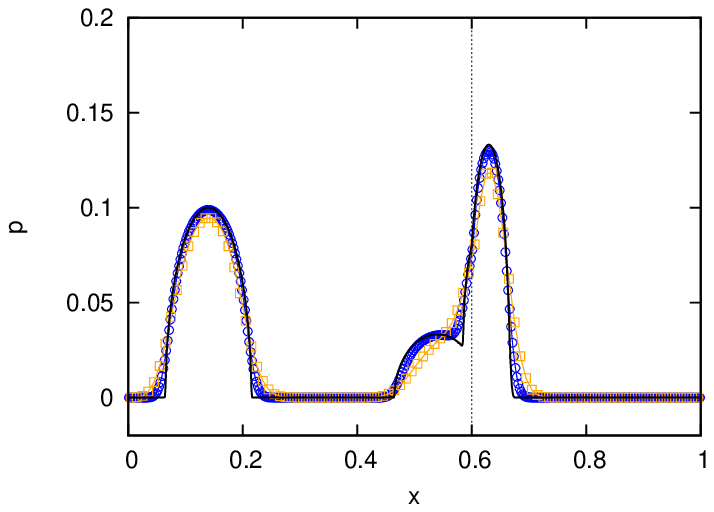}
    \includegraphics[width=0.49\textwidth]{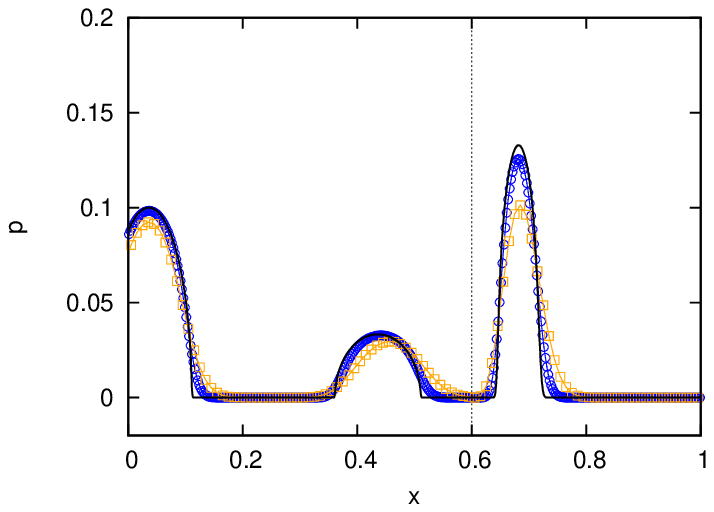}
    \includegraphics[width=0.49\textwidth]{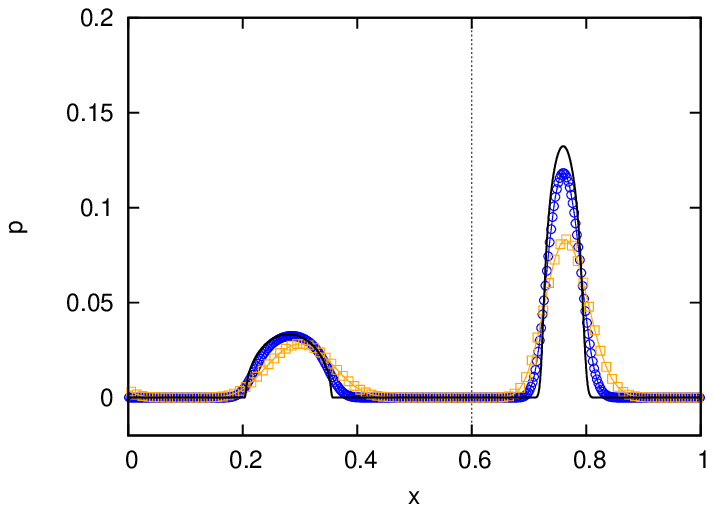}
    \end{center}
    \caption{The computed results ($\Delta x=0.01$
    ($\textcolor[rgb]{1,0.68,0.36}{-\circ-}$) and 
    $\Delta x=0.0025$ ($\textcolor[rgb]{0,0,1}{-\circ-}$)) converge
    to the reference solution (black line) as seen in the
    snapshots ($t=0.104$ (top-left), $t=0.26$ (top-right), 
    $t=0.364$ (bottom-left) and $t=0.52$ (bottom-right)).
    \label{fig:res_ac1}}
 \end{figure}

While the augmented scheme converges to the reference solution, it has
a higher dispersion error that manifests itself in a phase error of
the bumps propagating through the domain. As the mesh resolution is
refined, the dispersion error is rapidly reduced.  Fig. \ref{res_ac2}
shows a detail of the small reflected hump at $t=0.52$ for different
mesh sizes. It must be noted that the reference wave propagation
scheme does not show any dispersion error (i.e., the numerical
representation of the bump is symmetric which respect to the center of
the exact bump).
 
 \begin{figure} \centering
\includegraphics[width=0.49\textwidth]{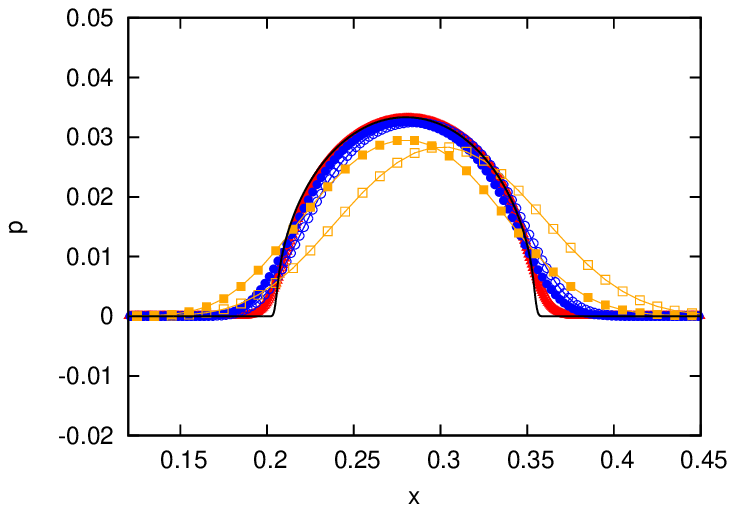}
\includegraphics[width=0.49\textwidth]{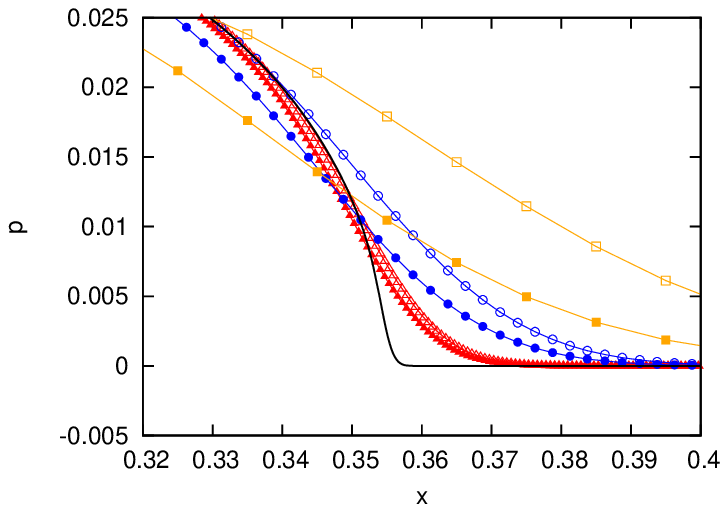}
\caption{The computed results by the augmented solver ($\Delta x=0.01$
($\textcolor[rgb]{1,0.68,0.36}{-\square-}$), $\Delta x=0.0025$
($\textcolor[rgb]{0,0,1}{-\circ-}$) and $\Delta x=0.00125$
($\textcolor[rgb]{1,0,0}{-\triangle-}$)) and reference solver (filled symbols)
show a phase error compared to the reference solution (black line) at $t=0.52$,
which rapidly reduces as the grid is refined.\label{res_ac2} } \end{figure}

\subsection{Hyperbolic heat equation}

The 1D heat equation in its original parabolic form reads
\begin{equation}\label{eq:par-heat}
    c \left( x \right) \rho \left( x \right) \frac{\partial u}{\partial t} + \frac{\partial}{\partial x} \left( -k \left( x \right) \frac{\partial u}{\partial x} \right) = \phi,
\end{equation}
where $u$ is the temperature, $c$ is the specific heat capacity,
$\rho$ is the density, $k$ is the thermal conductivity of the material
and $\phi$ is a heat source. The equation can be hyperbolized using
Cattaneo's relaxation approach \cite{Cattaneo_1958} as
\begin{equation}\label{eq:hyp-heat}
    \left\{
    \begin{aligned}
   c \left(x\right) \rho\left( x \right) 
   \frac{\partial u}{\partial t} 
   + \frac{\partial q}{\partial x} = \phi,\\
    \varepsilon \frac{\partial q}{\partial t} + k \left( x \right) 
    \frac{\partial u}{\partial x} = - q.
    \end{aligned}
    \right.
\end{equation}
Here, we introduce $\varepsilon$ as the relaxation time, and $q$ as
the flow of heat energy per unit area. For $\varepsilon \rightarrow
0$, the hyperbolic form in Eq. (\ref{eq:hyp-heat}) converges to the
original form in Eq. (\ref{eq:par-heat}). See \ref{app:heat} for the
eigenvalues and eigenvectors of the system.

\begin{proposition}\label{prop:varepsilon}
For a given cell size $\Delta x$ and a first order accurate numerical
solver, the value of $\varepsilon$ can be constrained as
\begin{equation}
\left\{
\begin{aligned}
    \varepsilon < \frac{O\left(1\right) \Delta x}{K_1},\\
    K_1 = \frac{1 - 2^{-\frac{1}{2}}}{2^{\frac{1}{2}} - 1}.
\end{aligned}
\right.
\end{equation}
Here, $O\left(1\right)$ denotes the error of a
first order numerical scheme.
\end{proposition}

\begin{proof}
Omitted here. See \cite{Montecinos_2014} for a detailed proof.
\end{proof}

\begin{proposition}\label{prop:heat-solution}
Let $f:\, \mathbb{R} \rightarrow \mathbb{R}$ be a function with
compact support on a domain $\Omega$ that satisfies $\frac{\partial f
  \left( x \right)}{\partial x} = \mathscr{M} f^{-1} \left(x\right)$,
with $f^{-1}: \mathbb{R} \rightarrow \mathbb{R}$, and $f\left(x\right)
f^{-1}\left(x\right)=1$.  Then, if $\phi=0$, $c\left(x\right) =
\rho\left(x\right) = 1$, and $k\left(x\right) = f$, then $f$ is a
steady state solution of Eq. (\ref{eq:par-heat}) as well as
Eq. (\ref{eq:hyp-heat}).
\end{proposition}

\begin{proof}
Since
\begin{equation}
    \frac{\partial u}{\partial x} = \mathscr{M} f^{-1},
\end{equation}
we obtain
\begin{equation}\label{eq:steady-heat-q}
   q = - k \left( x \right) \frac{\partial u}{\partial x} = - \mathscr{M},
\end{equation}
which yields a zero flux in the first equation of the system in Eq.
(\ref{eq:par-heat}) and Eq. (\ref{eq:hyp-heat}). Hence, $u$ does not
change in time and Prop. \ref{prop:heat-solution} is thus proved.
\end{proof}

\begin{corollary}\label{cor:sine}
Consider Eq. (\ref{eq:hyp-heat}) in an infinite 1D domain. If $\phi=0$
and $k \left( x \right) = \frac{\mathscr{M}}{\mathscr{W}\text{sin}
  \left( x \right) + \mathscr{C}}$, then $u\left( x \right) = -
\mathscr{W}\text{cos} \left(x \right) + \mathscr{C} x$ is a steady
state solution of the system for $\forall{\left(\mathscr{M},
  \mathscr{W},\mathscr{C}\right)} \in \mathbb{R}: \vert \mathscr{C}
\vert > \vert \mathscr{W}\vert $.
\end{corollary}

\begin{corollary}\label{prop:heat-solution-disc}
Consider Eq. (\ref{eq:hyp-heat}) in the domain $\Omega=[a,b]$. Let
$\phi=0$ and
\begin{equation}\label{mediadef00}
 k\left(x\right) = 
 \begin{cases}
     k_1 & \text{if } x \leq \frac{b+a}{2}-\frac{\delta}{2},\\
     \tilde{k} & \text{if } \frac{b+a}{2} - \frac{\delta}{2} < x < \frac{b+a}{2} + \frac{\delta}{2},\\
      k_2 & \text{if } x \geq \frac{b+a}{2}+\frac{\delta}{2},
 \end{cases}
\end{equation}
with $k_1,k_2 \in \mathbb{R}$ being nonzero constant conductivities,
$\tilde{k}=(k_1+k_2)/2$ and $\delta \in \mathbb{R}$ being the
conductivity and the thickness of the transition layer.  Then $q\left(
x \right) = \mathscr{M}$ and
\begin{equation}
  u(x)=
  \begin{cases}
  -\frac{ \mathscr{M}}{k_2} x & \text{if } x \leq \frac{b+a}{2}-\frac{\delta}{2},\\
   - \mathscr{M}\left(\frac{1}{k_1}x+\frac{k_1-k_2}{k_1 k_2}x^* + \frac{1}{k_1}\delta 
   - \frac{2}{\tilde{k}}\delta\right) & \text{if } x \geq \frac{b+a}{2}+\frac{\delta}{2},\\
  \end{cases}
\end{equation}
with $x^*=\frac{b+a}{2}-\frac{\delta}{2}$ is a weak solution 
of Eq. (\ref{eq:hyp-heat}) at the steady state, provided that a 
linear averaging of the coefficient matrix in Eq. (\ref{eq:hyp-heat}) is used.
\end{corollary}

\begin{corollary}\label{cor:heat-solution-disc-exact}
In the limit when $\delta \rightarrow 0$, the expressions in  Cor. 
\ref{prop:heat-solution-disc} represents the solution of a pure discontinuous 
transition in the medium properties (e.g., the conductivity). Such solution reads 
$q\left( x \right) = \mathscr{M}$ and 
 \begin{equation}
  u\left(x\right) =
  \begin{cases}
    -\frac{ \mathscr{M}}{k_2}x & \text{if } x  \leq \frac{b+a}{2},\\
    - \mathscr{M}\left(\frac{1}{k_1}x+\left(\frac{k_1-k_2}{k_1 k_2}\right) 
    \left(\frac{b+a}{2}\right) \right) & \text{if } x \geq \frac{b+a}{2}.\\
  \end{cases}
\end{equation}
\end{corollary}

\begin{corollary}\label{cor:sol}
When considering the numerical resolution of Eq. (\ref{eq:hyp-heat})
by means of an augmented scheme and $\delta=\Delta x$, then the
solution in Cor. \ref{prop:heat-solution-disc} corresponds to the
numerical solution provided by the scheme under steady state. This
means that, in the case of a pure discontinuous transition, the
numerical scheme will artificially enforce the presence of a
transition layer of size $\delta=\Delta x$.  As the grid is refined,
the size of the transition layer is reduced and tends to zero, thus
approaching the exact solution of the pure discontinuous transition.
\end{corollary}

\begin{proposition}\label{prop:heat-solution-source}
Consider Eq. (\ref{eq:hyp-heat}) in the domain $\Omega=[a,b]$.  If
$\phi,k \in \mathbb{R}$ are nonzero constants, then $q\left( x \right)
= q(a) + \phi (x-a)$ and $u\left( x \right) = u(a) -
\frac{q(a)}{k}(x-a) - \frac{\phi}{2k}(x-a)^2$ is a steady solution of
Eq. (\ref{eq:hyp-heat}), with $q(a)$ the heat flux at the inlet and
$u(a)$ the temperature at the inlet.
\end{proposition}

\begin{proof}
The given expressions for the discharge and temperature in 
Prop. \ref{prop:heat-solution-source} satisfy the system in 
Eq. (\ref{eq:hyp-heat}), namely the equations  
 \begin{equation}
    \left\{
    \begin{aligned}
   &\frac{\partial q}{\partial x} = \phi,\\
    &k  \frac{\partial u}{\partial x} = - q.
    \end{aligned}
    \right.
\end{equation}
 Prop. \ref{prop:heat-solution-source} is thus proved.
\end{proof}

In the following, we show results obtained by the augmented solver for
the hyperbolic heat equation at the steady state.  For contrast, we
compute results with a deliberately not well-balanced wave propagation
scheme, which uses a centered integration of the source term.

We note that high-order accurate steady state solutions can be
computed using wave propagation algorithms (e.g., by following the
approaches in \cite{LeVeque_1986} (fractional time stepping) and
\cite{LeVeque_2011} (path-conservative approach)). The presented
results are for illustration purposes only, and do not represent the
state of the art wave propagation solvers.

\subsubsection{Steady solution considering a constant conductivity}

This case is devoted to the resolution of the hyperbolized heat
equation with $\rho \left( x \right) =c \left(x \right) = 1$ and
$\phi=0$, on the domain $x \in [a, b]$, where $a=0$ and $b=10$. The
conductivity, $k\left(x\right)$, is set constant. Following
Cor. \ref{cor:sine}, we impose $\mathscr{M} = 1$, $\mathscr{W} = 0$
and $\mathscr{C} = 2$.  Initial conditions for $u$ and $q$ are applied
following Cor.  \ref{cor:sine}. Fixed Dirichlet boundary conditions
are applied for both heat flux and temperature. The heat flux is set
upstream and the temperature downstream, as
 \begin{equation}
   q_{in} = - 1,\quad u_{out} = \mathscr{C}b.
 \end{equation}
 
The augmented solver reproduces the exact solution independent of the
grid resolution. Both the nonconservative products and the source term
are constant, and the linear approximation of the integral of these
terms is exact, see Prop. \ref{prop:exact-schemes}.  The not
well-balanced scheme is not able to reproduce the analytical reference
solution, neither for the temperature nor the heat flux. See
Fig. \ref{res_hyp_heat1}, which shows the computed solution of both
schemes after $30,000$ time steps, using a CFL number of $0.8$. The
$L_\infty$-norm of the error of the augmented scheme is in the range
of machine accuracy, see Tab.  \ref{tabla:conv1}.

\begin{figure}
\centering \includegraphics[width=0.49\textwidth]{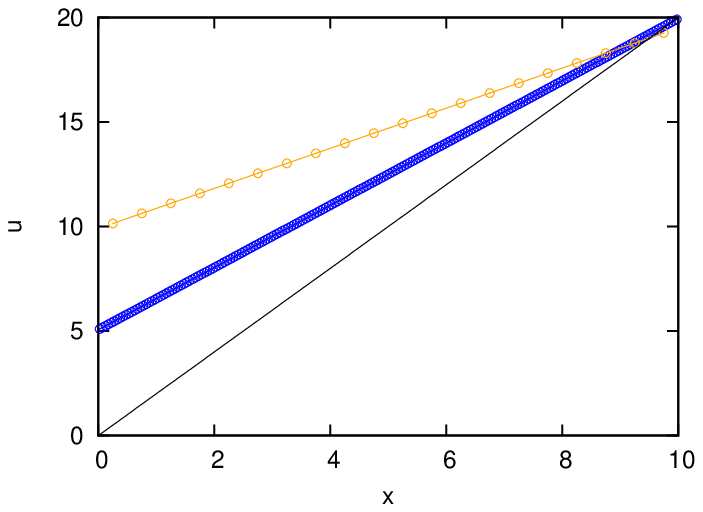}
\includegraphics[width=0.49\textwidth]{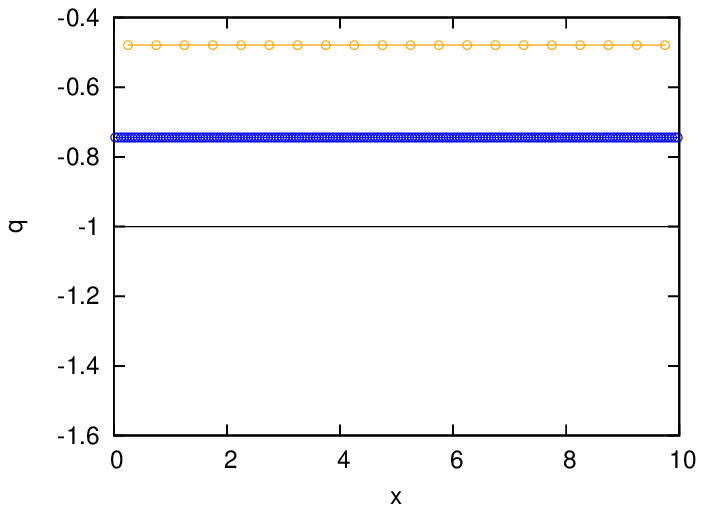}
\includegraphics[width=0.49\textwidth]{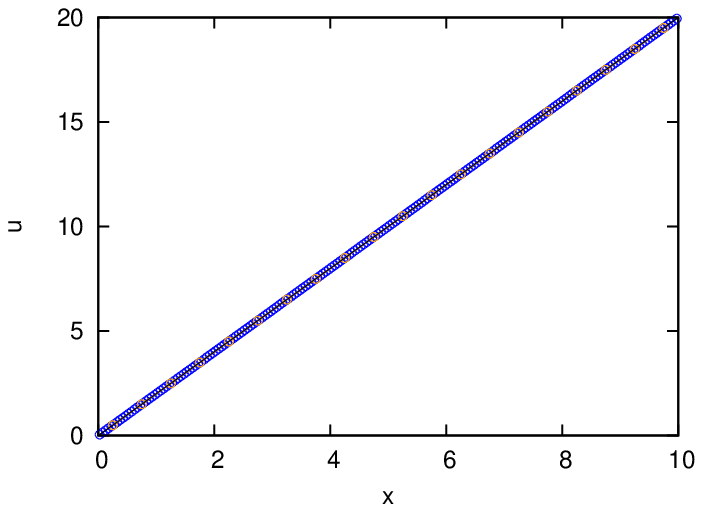}
\includegraphics[width=0.49\textwidth]{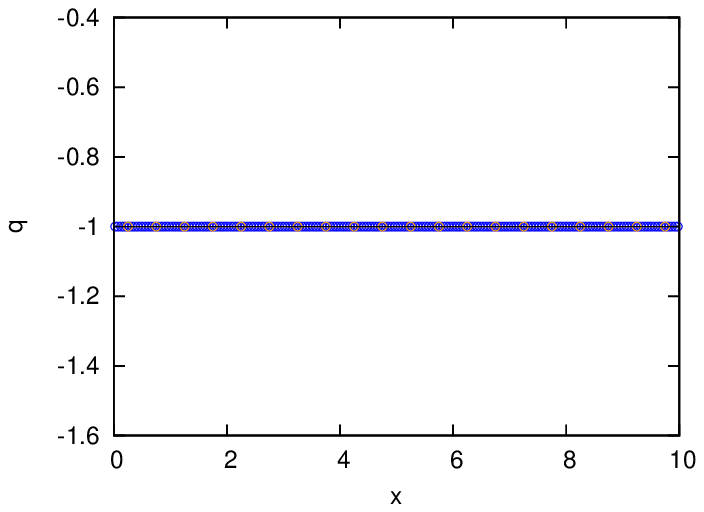}
\caption{Constant conductivity: The not well-balanced scheme (top)
  fails to reproduce the reference solution at the steady state (black
  line), but the augmented scheme (bottom) accurately reproduces
  it. (Numerical solution computed using $\Delta x=0.5$
  ($\textcolor[rgb]{1,0.68,0.36}{-\circ-}$) and $\Delta x=0.05$
  ($\textcolor[rgb]{0,0.0,1}{-\circ-}$)).\label{res_hyp_heat1}}
\end{figure}
 
\begin{table}
\centering
\begin{tabular}{ccccc}
 &  \multicolumn{2}{c}{Unbalanced} & \multicolumn{2}{c}{Augmented}\\
\midrule
{Mesh}   & $L_{\infty}(u)$   & $L_{\infty}(q)$    & $L_{\infty}(u)$   & $L_{\infty}(q)$\\
\midrule
$\Delta x=0.5$   &  9.65e+00 & 5.21e-01   & 2.22e-16  &  3.33e-16\\
$\Delta x=0.05$   &  5.05e+00  & 2.56e-01   & 8.88e-16  & 1.38e-14 \\
\bottomrule
\end{tabular}
\caption{Constant conductivity: $L_{\infty}$ error norms for the numerical solution computed by 
the not well-balanced algorithm and the augmented scheme.}
\label{tabla:conv1}
\end{table}

\subsubsection{Steady solution considering a smoothly varying conductivity }

This case considers again the resolution of the hyperbolized
homogeneous heat equation (i.e., $\phi=0$) with $\rho \left( x \right)
=c \left(x \right) = 1$, but in this case a varying conductivity is
chosen. $k\left(x\right)$ is set as presented in Cor. \ref{cor:sine},
with $\mathscr{M} = 1$, $\mathscr{W} = 1.8$ and $\mathscr{C} = 2$. The
computational domain is given by $x \in [a, b]$, where $a=0$ and
$b=10$. Initial conditions for $u$ and $q$ are applied following
Prop. \ref{cor:sine}. Fixed Dirichlet boundary conditions are applied
for both heat flux and temperature. The heat flux is set upstream and
the temperature downstream, as follows
\begin{equation}
   q_{in} = - 1,\quad u_{out} = - \mathscr{W}\mathrm{cos} \left(b
   \right) + \mathscr{C} b
\end{equation}

Using a CFL number of $0.8$, steady state is reached after $500,000$
time steps.  The augmented solver accurately reproduces the reference
solution for heat flux independent of grid resolution (as anticipated
in Cor. \ref{cor:exact-source}), but is unable to reproduce the
reference solution for the temperature. This is due to the temperature
and conductivity being sinusoidal functions that do not satisfy the
condition in Prop. \ref{prop:exact-schemes}, see the $L_\infty$-norms
for the errors in Tab. \ref{tabla:conv2}. Despite this fact, the
augmented scheme is more accurate than the not well-balanced scheme in
recovering the steady state solution, see Fig. \ref{res_hyp_heat2}.

\begin{figure}
    \centering
    \includegraphics[width=0.49\textwidth]{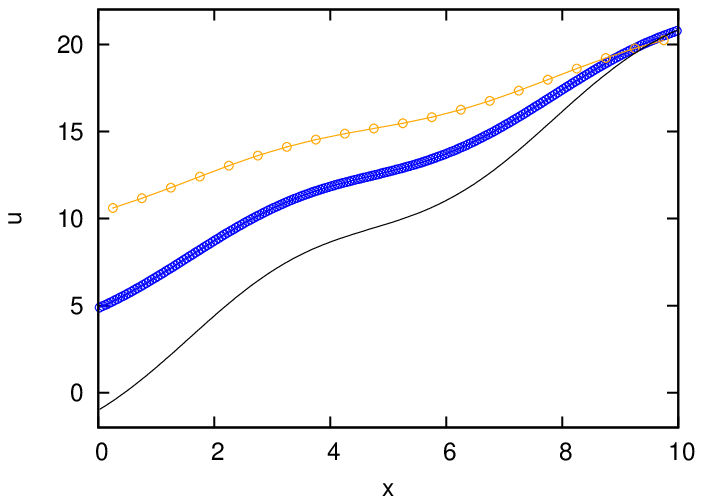}
    \includegraphics[width=0.49\textwidth]{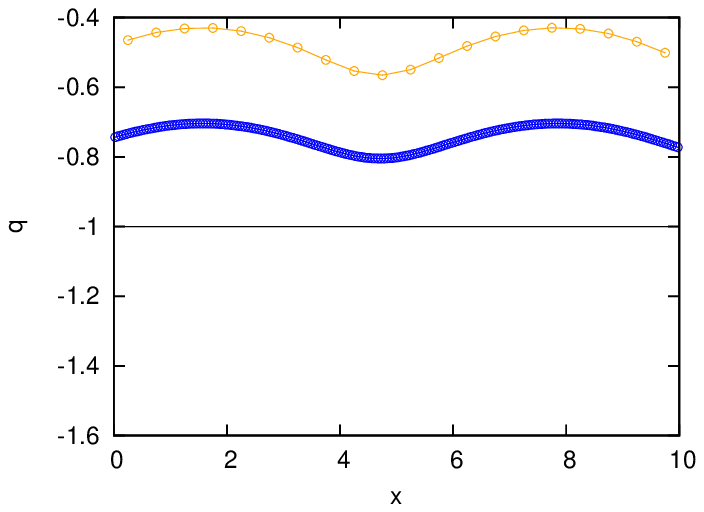}
    \includegraphics[width=0.49\textwidth]{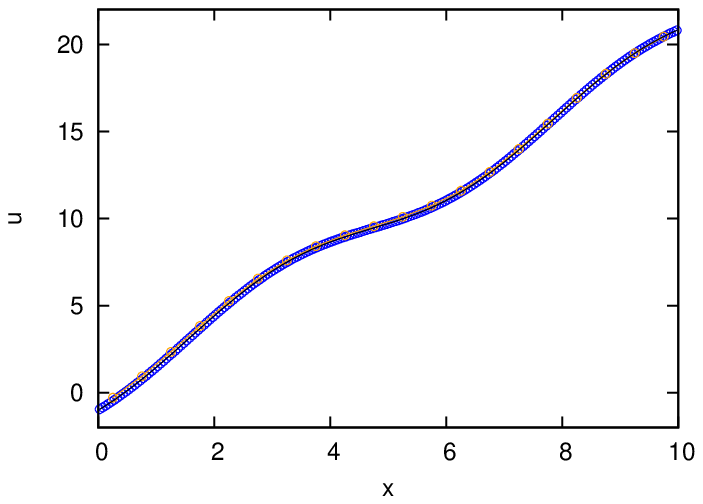}
    \includegraphics[width=0.49\textwidth]{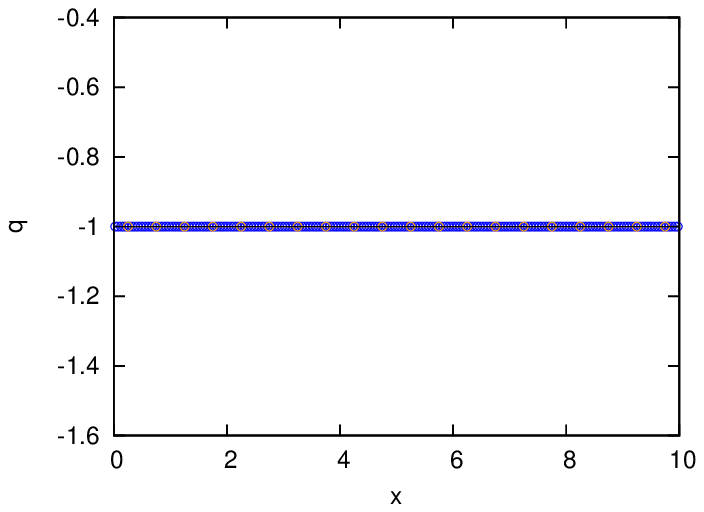}
    \caption{Smoothly varying conductivity: The not well-balanced
      scheme (top) fails to reproduce the reference solution at the
      steady state (black line), the augmented scheme (bottom)
      reproduces it more accurately. (Numerical solution computed
      using $\Delta x=0.5$ ($\textcolor[rgb]{1,0.68,0.36}{-\circ-}$)
      and $\Delta x=0.05$
      ($\textcolor[rgb]{0,0.0,1}{-\circ-}$)).\label{res_hyp_heat2}}
 \end{figure}
 
\begin{table}
\centering
\begin{tabular}{ccccc}
 &  \multicolumn{2}{c}{Unbalanced} & \multicolumn{2}{c}{Augmented}\\
\midrule
{Mesh}   & $L_{\infty}(u)$   & $L_{\infty}(q)$    & $L_{\infty}(u)$   & $L_{\infty}(q)$\\
\midrule
$\Delta x=0.5$   &  1.11e+01 & 5.71e-01   & 1.86e-01   &  5.66e-14\\
$\Delta x=0.05$   & 5.85e+00  & 2.96e-01   & 2.02e-03  & 7.56e-14 \\
\bottomrule
\end{tabular}
\caption{Smoothly varying conductivity: $L_{\infty}$ error norms for
  the numerical solution computed by the not well-balanced algorithm
  and the augmented scheme.}
\label{tabla:conv2}
\end{table}
 
\subsubsection{Steady solution considering a piecewise constant conductivity}

In this case, the hyperbolized heat equation with $\phi=0$, $\rho
\left( x \right) =c \left(x \right) = 1$ is solved considering a
piecewise constant conductivity, given by two constant values of
conductivity separated by a discontinuity as
\begin{equation}\label{mediadef1}
 k \left(x\right) =
\begin{cases}
    k_1 = 1 & \text{if } x < \frac{b+a}{2},\\
    k_2 = 4 & \text{if } x > \frac{b+a}{2}.\\
\end{cases}
\end{equation}
This test case represents the limit situation when making $\delta
\rightarrow 0$ in the exact solution of
Cor. \ref{prop:heat-solution-disc}. The computational domain is given
by $x \in [a, b]$, where $a=0$ and $b=10$. Fixed Dirichlet boundary
conditions are applied for both heat flux and temperature.  The heat
flux is set upstream and the temperature downstream as
\begin{equation}
  q_{in} = - \mathscr{M},\quad u_{out} = -\frac{ \mathscr{M}}{k_2} b
\end{equation}
where $\mathscr{M} = 1$.
 
Steady state is obtained after $30,000$ time steps with a CFL number
of $0.8$. The equilibrium solution provided by the augmented scheme
converges to the solution in Cor. \ref{prop:heat-solution-disc}, where
$\delta = \Delta x$, and converges with first order accuracy to the
exact analytical solution in Cor. \ref{cor:heat-solution-disc-exact}
as the grid resolution is refined, as stated in Cor. \ref{cor:sol},
see the solution plot in Fig.  \ref{res_hyp_heat3} and the detail plot
of the domain around the conductivity jump in
Fig. \ref{res_hyp_heat31}. The $L_{\infty}$ error norm with respect to
the exact solution and the analytical discrete equilibrium solution
shows that the model converges to the equilibrium solution with
machine accuracy, but not to the exact analytical solution, see
Tab. \ref{tabla:conv3}.

 \begin{figure}
	 \centering
 \includegraphics[width=0.49\textwidth]{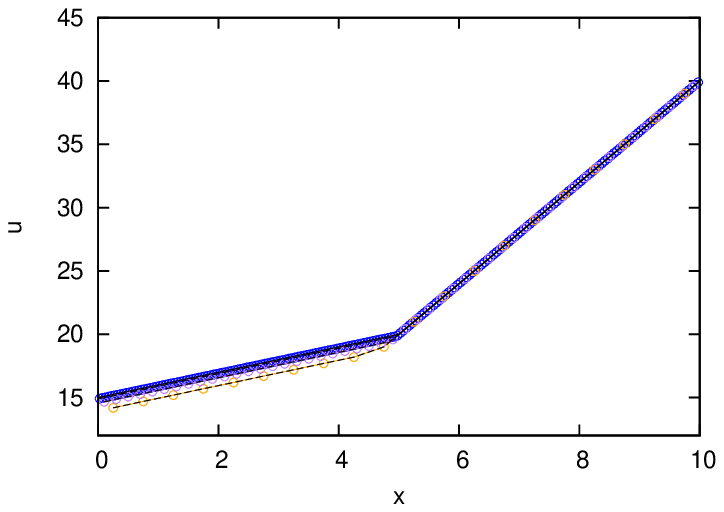}
 \includegraphics[width=0.49\textwidth]{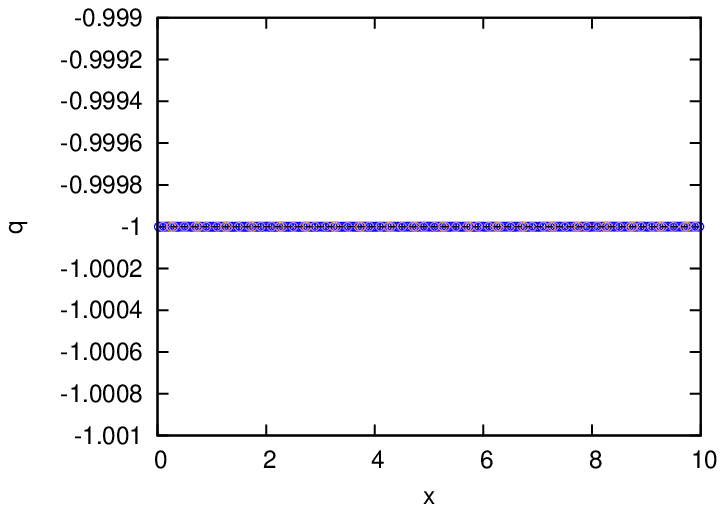}
	 \caption{Piecewise constant conductivity: The augmented
           scheme ($\Delta x=0.5$
           ($\textcolor[rgb]{1,0.68,0.36}{-\circ-}$), $\Delta x=0.2$
           ($\textcolor[rgb]{0.75,0,1}{-\circ-}$), $\Delta x=0.05$
           ($\textcolor[rgb]{0,0.0,1}{-\circ-}$)) converges to the
           equilibrium solution (dashed line) with machine accuracy,
           but fails to reproduce the analytical solution (solid black
           line).\label{res_hyp_heat3}}
 \end{figure}
 
\begin{figure}
	 \centering
 \includegraphics[width=0.75\textwidth]{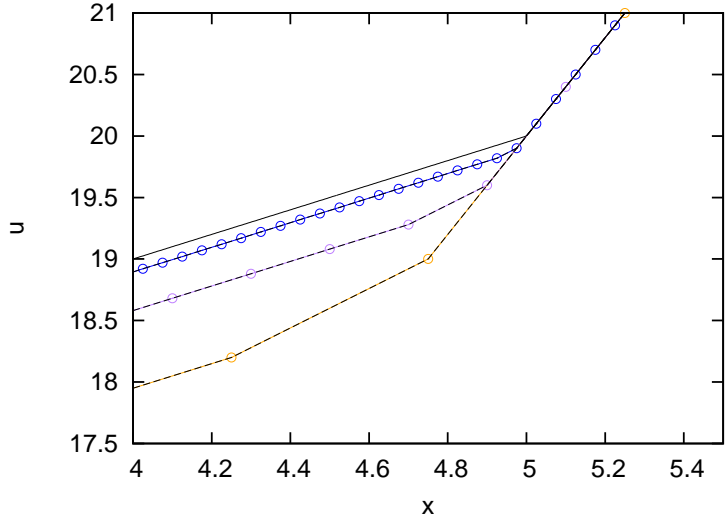}
	 \caption{Piecewise constant conductivity: Detail of the
           solution for temperature, computed by the augmented scheme
           ($\Delta x=0.5$ ($\textcolor[rgb]{1,0.68,0.36}{-\circ-}$),
           $\Delta x=0.2$ ($\textcolor[rgb]{0.75,0,1}{-\circ-}$),
           $\Delta x=0.05$ ($\textcolor[rgb]{0,0.0,1}{-\circ-}$))
           shows the deviation between model results and the exact
           analytical solution (black solid line). The equilibrium
           solution (dashed line) is always captured with machine
           accuracy. \label{res_hyp_heat31}}
\end{figure}
 
\begin{table}
\centering
\begin{tabular}{ccccc}
 & \multicolumn{2}{c}{Exact ($\delta =0$)} &
  \multicolumn{2}{c}{Equilibrium ($\delta =\Delta x$)}\\ \midrule
              {Mesh} & $L_{\infty}(u)$ & $L_{\infty}(q)$ &
              $L_{\infty}(u)$ & $L_{\infty}(q)$\\ \midrule $\Delta
              x=0.5 $ & 1.05e+00 & 4.44e-16 & 1.00e-16 &
              4.44e-16\\ $\Delta x=0.2 $ & 4.20e-01 & 5.11e-15 &
              1.78e-15 & 5.11e-15 \\ $\Delta x=0.05$ & 1.05e-01 &
              1.58e-14 & 1.78e-15 & 1.58e-14 \\ \bottomrule
\end{tabular}
\caption{Piecewise constant conductivity: $L_{\infty}$ error norms of
  the numerical solution with respect to the exact solution and the
  equilibrium solutions computed by the augmented scheme.}
\label{tabla:conv3}
\end{table}

\subsubsection{Steady solution considering an external heat source}

The hyperbolized heat equation with a constant heat source $\phi=0.5$
is considered inside the computational domain $x \in [a, b]$, where
$a=0$ and $b=10$. The product of density and specific heat capacity is
set to $\rho \left( x \right) c \left(x \right) = 0.5$. The
conductivity is set constant, $k\left(x\right)=3$, yielding to the
solution provided in Prop. \ref{prop:heat-solution-source}, where the
discharge at the inlet is set as $q(a)=-1$. Initial conditions for $u$
and $q$ are applied following such solution in
Prop. \ref{prop:heat-solution-source}. Fixed Dirichlet boundary
conditions are applied for both heat flux and temperature. The heat
flux is set upstream and the temperature downstream, as follows
\begin{equation}
   q_{in} \equiv q(a) =  - 1 ,\quad u_{out} = -\frac{\phi}{2k}b^2 - q(a)\frac{b}{k}
\end{equation}

Steady state is obtained after $500,000$ time steps, using a CFL number of $0.9$.
As anticipated in Cor. \ref{cor:exact-source}, the augmented scheme provides the 
exact solution independent of the grid, ensuring the equality in Eq. (\ref{eq:cond_discrete}),
see Fig. \ref{res_hyp_heat4}, and Tab. \ref{tabla:conv2}.

\begin{figure}
\centering
\includegraphics[width=0.49\textwidth]{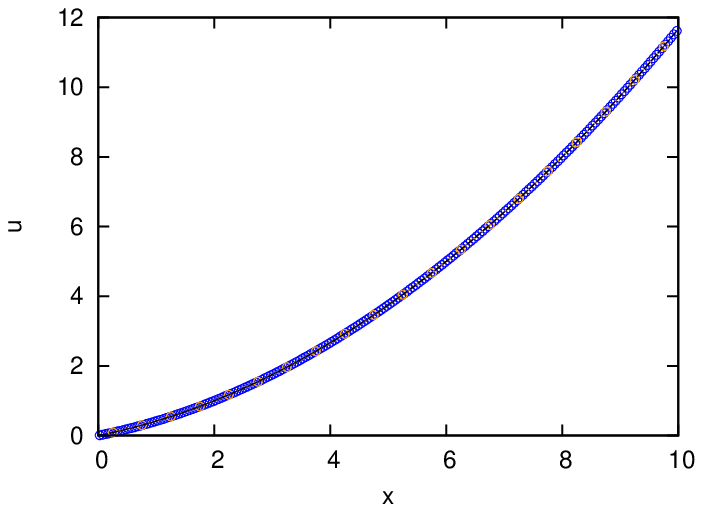}
\includegraphics[width=0.49\textwidth]{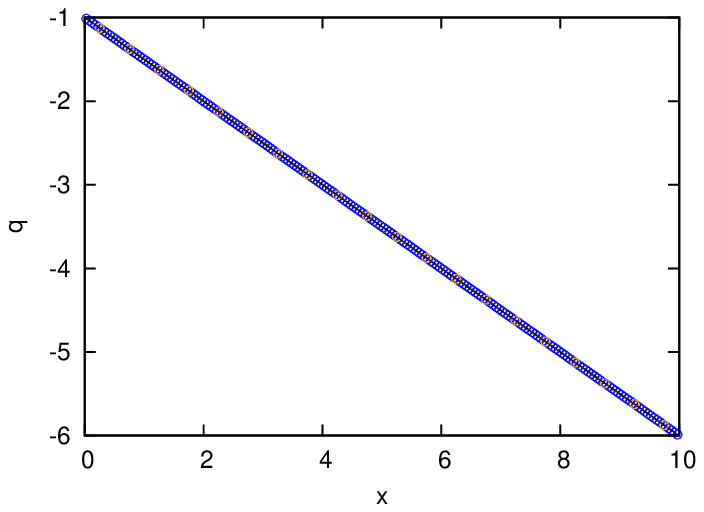}
\caption{External heat source: Augmented scheme ($\Delta x=0.5$
  ($\textcolor[rgb]{1,0.68,0.36}{-\circ-}$) and $\Delta x=0.05$
  ($\textcolor[rgb]{0,0.0,1}{-\circ-}$)) converges to the exact
  solution (solid line) with machine accuracy.\label{res_hyp_heat4}}
\end{figure}
 
\begin{table}
\centering
\begin{tabular}{ccc}
\midrule
{Mesh}      & $L_{\infty}(u)$   & $L_{\infty}(q)$\\
\midrule
$\Delta x=0.5$   & 1.00e-16   &  4.44e-16 \\
$\Delta x=0.05$  & 3.55e-15  &   1.33e-14 \\
\bottomrule
\end{tabular}
\caption{External heat source: $L_{\infty}$ error norms for the numerical 
solution computed by the augmented scheme.}
\label{tabla:conv4}
\end{table}

\subsubsection{Transient solution for a Riemann problem  with constant conductivity} 

This test considers the resolution of a RP for the hyperbolic heat
equation in Eq. (\ref{eq:hyp-heat}) with the IC given as
\begin{equation}\label{eq:rp1_ini}
\left\{
\begin{aligned}
   &\accentset{\circ}{u}\left(x\right) = 
   \begin{cases}
      	-1 & \text{if }  x < 5, \\
		1  & \text{if }  x > 5, \\
    \end{cases} \\ 
	& \accentset{\circ}{q}\left(x\right) = 0,
\end{aligned}
\right.
\end{equation}
and considering a constant conductivity $k(x)=0.05$. No source term is
considered, $\phi=0$, and $\rho \left( x \right) =c \left(x \right) =
1$. The computational domain is set as $\Omega=[0,10]$ and the
solution is computed at $t=2$ and $t=5$ using $\Delta x=0.5$ and
$\Delta x = 0.1$ by means of the augmented scheme. The CFL number is
set to 0.5 and $\varepsilon$ is set according to
Prop. \ref{prop:varepsilon}.

For this problem configuration it is possible to find an analytical
condition, which reads:
\begin{equation}\label{eq:analytical_rp1}
u(x,t)= \text{erf} \left( \frac{x-5}{2\sqrt{k t}} \right) \,, \quad
q(x,t)=-k \frac{e^{-(x-5)^2/(4 k t))}}{\sqrt{k \pi t}};
\end{equation}

The numerical results evidence that the augmented scheme accurately
captures the transient evolution of the temperature and heat flux and
converges to the exact solution when the mesh is refined. The
numerical solution is depicted in Fig. \ref{res_hyp_heat_tr1}.

\begin{figure}
\centering
\includegraphics[width=0.49\textwidth]{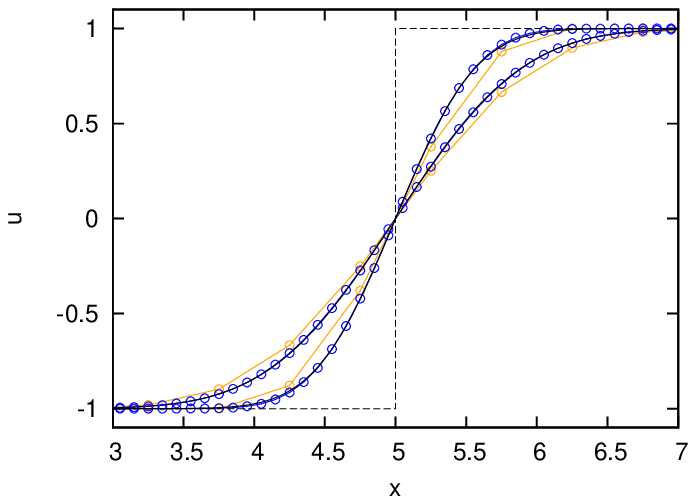}
\includegraphics[width=0.49\textwidth]{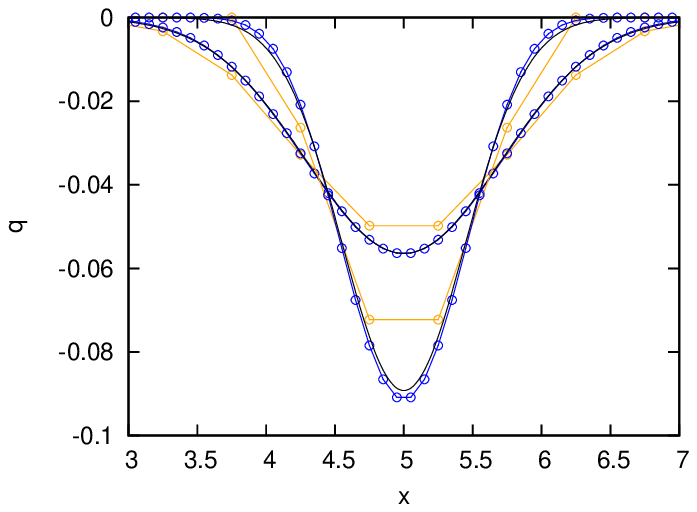}
\caption{RP with constant conductivity: Numerical solution provided by
  the augmented scheme ($\Delta x=0.5$
  ($\textcolor[rgb]{1,0.68,0.36}{-\circ-}$) and $\Delta x=0.1$
  ($\textcolor[rgb]{0,0.0,1}{-\circ-}$)) accurately captures the exact
  solution (solid line) in the transient case. Dashed line shows the
  IC. \label{res_hyp_heat_tr1}}
\end{figure}

\subsubsection{Transient solution for a Riemann problem with piecewise 
constant conductivity} 

This test considers the resolution of a RP for the heat equation in
Eq. (\ref{eq:par-heat}) with the initial condition in
Eq. (\ref{eq:rp1_ini}).  No source term is considered (i.e.,
$\phi=0$). Medium properties are set to $\rho \left( x \right) =c
\left(x \right) = 1$. Two different piecewise constant distributions
of conductivity are considered. The first case, referred to as RP-a,
uses
\begin{equation}\label{eq:rp2_a}
  k(x) = 
  \begin{cases}
    0.1  & \text{if } x < 5, \\
    0.01 & \text{if } x > 5, 
  \end{cases}
\end{equation}
whereas the second case, referred to as RP-b, uses
\begin{equation}\label{eq:rp2_b}
  k(x) =
  \begin{cases}
    0.1  & \text{if }  x < 5, \\
    0.05 & \text{if }  x > 5.
  \end{cases}
\end{equation}
 
The computational domain is set as $\Omega=[0,10]$ and the solution is
computed at $t=8$ using $\Delta x=0.5$ and $\Delta x = 0.1$ by means
of the augmented scheme. The CFL number is set to 0.5 and
$\varepsilon$ is set following Prop. \ref{prop:varepsilon}.

A numerical solution computed with the augmented scheme on a grid of
$\Delta x=0.001$ is used as reference solution for this case.  The
numerical results evidence that the augmented scheme converges to the
reference solution when the mesh is refined. The proposed scheme is
able to handle strong discontinuities of the conductivity coefficient
without generating spurious oscillations. The numerical solution for
these test cases is depicted in Fig. \ref{res_hyp_heat_tr2a} and
Fig. \ref{res_hyp_heat_tr2b}.

\begin{figure}
\centering
\includegraphics[width=0.49\textwidth]{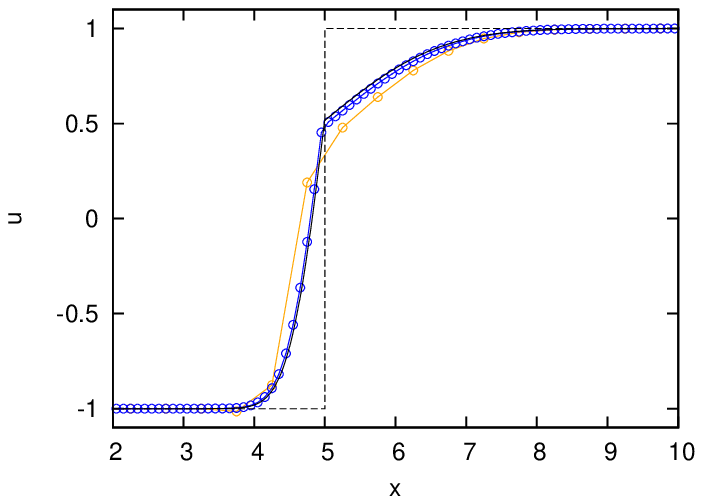}
\includegraphics[width=0.49\textwidth]{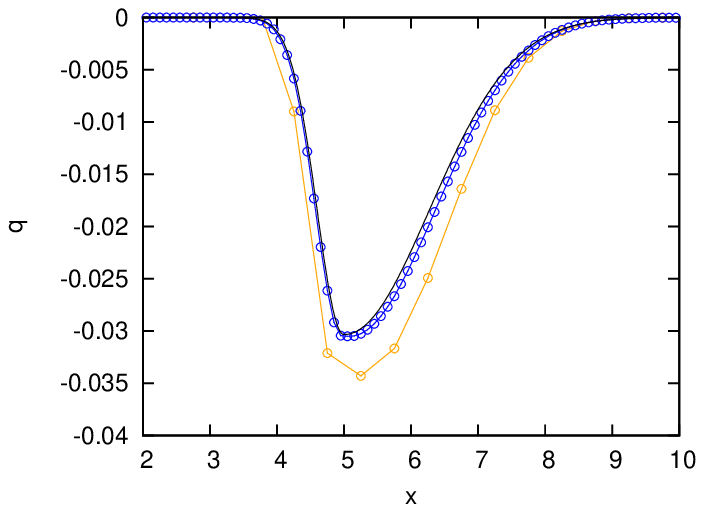}
\caption{RP-a with conductivity jump: Numerical solution provided by
  the augmented scheme ($\Delta x=0.5$
  ($\textcolor[rgb]{1,0.68,0.36}{-\circ-}$) and $\Delta x=0.1$
  ($\textcolor[rgb]{0,0.0,1}{-\circ-}$)) converges to the reference
  solution (solid line). Dashed line shows the
  IC.\label{res_hyp_heat_tr2a}}
\end{figure}

\begin{figure}
\centering
\includegraphics[width=0.49\textwidth]{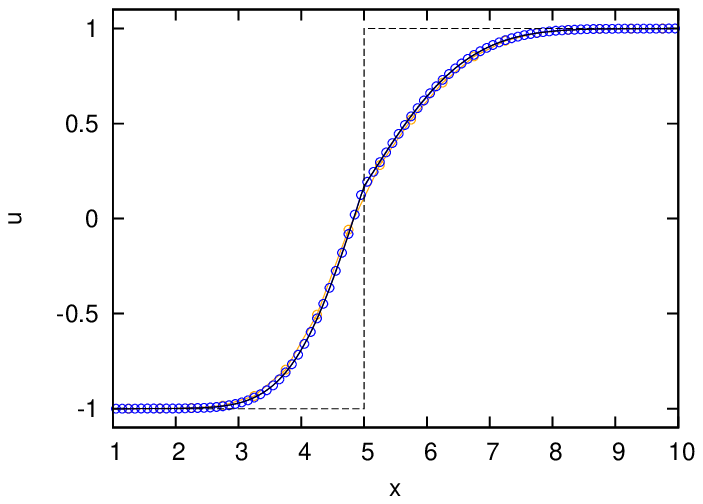}
\includegraphics[width=0.49\textwidth]{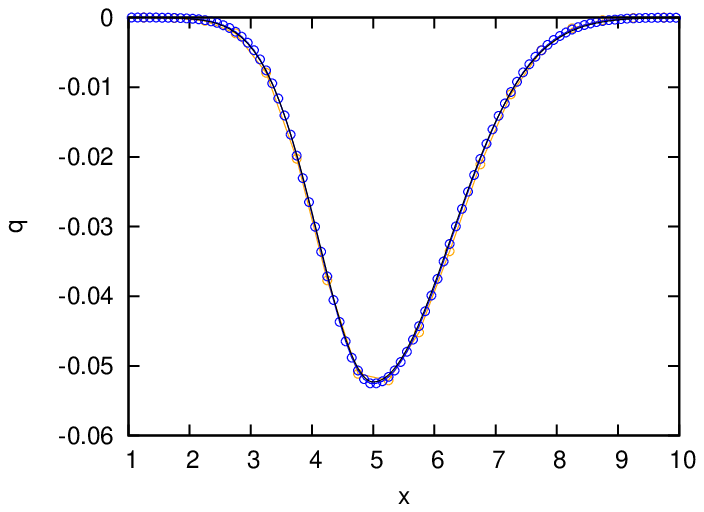}
\caption{RP-b with conductivity jump: Numerical solution provided by 
the augmented scheme ($\Delta x=0.5$ ($\textcolor[rgb]{1,0.68,0.36}{-\circ-}$) and $\Delta x=0.1$ 
($\textcolor[rgb]{0,0.0,1}{-\circ-}$)) converges to the reference
solution (solid line). Dashed line shows the IC.\label{res_hyp_heat_tr2b}}
\end{figure}

\section{Conclusions\label{sec:concl}}

This paper proposes a general framework to well-balance linear
noncon\-ser\-va\-tive hyperbolic systems with source terms, focusing
on the equilibrium properties of the numerical schemes.  Our schemes
are based on augmented Riemann solvers, which are revisited here using
a common framework that uses a geometric reinterpretation of source
terms. This geometric reinterpretation allows to handle arbitrary
sources in the same manner (i.e., they are introduced in the
definition of the Riemann problem as a singular source term).

It is proved that the schemes constructed within the proposed
framework are intrinsically well-balanced under linearity
conditions. In the problems where the source terms are linear and the
solution yields linear nonconservative products, the schemes will
reproduce the exact solution with machine precision. Otherwise, they
will converge to the exact solution with the order of accuracy chosen
for the discretization (e.g., first order in this work). It must be
noted that the schemes could be extended to higher order of accuracy
using traditional TVD, ENO or WENO reconstructions in space
\cite{harten_1987,liu_1994} and Runge-Kutta or ADER time-stepping
\cite{Shu1998, Schwartzkopff_2002, navasmontilla_2016, dumbser2013}.

In the presented test cases, we show that the proposed methods
accurately capture the transient evolution of the waves, though they
involve a higher dispersion error than the traditional wave
propagation algorithm \cite{LeVeque_1997}. The hyperbolized heat
equation allows a very complete exercise of the equilibrium properties
of the proposed methods under a wide variety of conditions (e.g.,
spatial variation of the medium conductivity, presence of a heat
source term with a possible variation in space, and presence of
discontinuities).  The proposed methods have been validated using as
benchmark a complete set of test cases, involving transient and steady
solutions, for which the analytical solutions have been presented. The
numerical results evidence that the proposed scheme is able to provide
a time-accurate resolution of transient cases. For steady-state
problems, the schemes reproduce the exact discharge with machine
accuracy for all cases provided a constant or linear heat source
term. The temperature is generally approximated with first order of
accuracy. In the case when the system in Eq.  (\ref{eq:systemnDim})
satisfies linearity conditions, the scheme also reproduces the exact
temperature with machine accuracy.

The schemes also evidence a good performance (i.e., convergence with
mesh refinement $L_{\infty} = \mathcal{O}(\Delta x)$), when dealing
with discontinuities in the medium properties. The generation of an
artificial transition layer in the numerical solution, in the sense of
the 3-shock representation of discontinuous solutions in nonlinear
systems \cite{navasmontilla_2019} is observed. The size of such
transition layer and the value of the variables in it have been
analytically derived and numerically confirmed for the hyperbolized
heat equation.

The proposed approach is a reliable tool to provide an accurate
numerical solution to heat diffusion problems in heterogeneous media
in presence of different source terms. In presence of heat sources, it
allows to exactly (or accurately, if the source function is nonlinear
and has to be numerically approximated) preserve the balance of heat
fluxes.  These properties make the proposed schemes a useful tool for
realistic engineering problems such as the simulation of phase-change
problems, refer to \cite[among others]{Voller_1987, Shatikian_2005,
  Shatikian_2007}, and may represent an alternative to more
traditional schemes such as the finite difference methods for
instance.

The methods in this work can be easily extended to solve
convection-diffusion problems where the nonconservative products
coexist with the conservative contributions that arise from the
convective term. An exact balancing of those terms can be achieved in
the same way. The extension of the proposed methods to higher spatial
dimensions is straightforward following a dimensional splitting
approach.

\section*{Acknowledgments}
A. Navas-Montilla acknowledges the partial funding by Gobierno de
Arag\'on through the Fondo Social Europeo.  The authors thank
R. J. LeVeque and K. Mandli for the insightful discussions on the
topic.

\bibliography{references}{}
\bibliographystyle{plain}

\appendix
\section{Eigenanalysis of the linear acoustic equation}\label{app:acoustic}

The linear acoustic equation in Eq. \ref{eq:acousticeq1D} can be written in
the form of Eq. \ref{eq:systemnDim} with
\begin{equation}
    U = \left[
    \begin{array}{c}
    p\\
    u
    \end{array} \right], \; \mathcal{A} = \left[ \begin{array}{c c} 
    0 & K \\
    \frac{1}{\rho} & 0
    \end{array} \right], \; S = \left[ \begin{array}{c}
    0\\
    0
    \end{array} \right].
\end{equation}
The eigenvalues of the coefficient matrix $\mathcal{A}$ read
\begin{equation}
    \lambda^1 = - \sqrt{\frac{K}{\rho}}, \; \lambda^2 = \sqrt{\frac{K}{\rho}},
\end{equation}
with corresponding right eigenvalue matrix $\mathcal{P}$, and its inverse as
\begin{equation}
    \mathcal{P} = \left[ \begin{array}{c c}
    - \frac{1}{\sqrt{K/\rho}} & \frac{1}{\sqrt{K/\rho}}\\
    1 & 1
    \end{array} \right], \quad
    \mathcal{P}^{-1} = \frac{1}{2} \left[ \begin{array}{c c}
    -\sqrt{\frac{K}{\rho}} & 1 \\
    \sqrt{\frac{K}{\rho}} & 1
    \end{array} \right].
\end{equation}
The system is strictly hyperbolic with $2$ real and distinct eigenvalues.

\section{Eigenanalysis of the hyperbolic heat equation}\label{app:heat}

\subsection{Original hyperbolic heat equation}

The hyperbolic heat equation in Eq. \ref{eq:hyp-heat} can be written in
the form of Eq. \ref{eq:systemnDim} with
\begin{equation}
    U = \left[
    \begin{array}{c}
    u\\
    q
    \end{array} \right], \; \mathcal{A} = \left[ \begin{array}{c c} 
    0 & r \\
    \frac{k}{\varepsilon} & 0
    \end{array} \right], \; S = \left[ \begin{array}{c}
    r\phi\\
    - \frac{q}{\varepsilon}
    \end{array} \right],
\end{equation}
and $c\left(x\right)  \rho\left(x\right) = 1/r$ for simplicity. The eigenvalues of 
the coefficient matrix $\mathcal{A}$ read
\begin{equation}
    \lambda^1 = - \sqrt{\frac{k r}{\varepsilon}}, \; \lambda^2 = \sqrt{\frac{k r}{\varepsilon}},
\end{equation}
with corresponding right eigenvalue matrix $\mathcal{P}$, and its inverse as
\begin{equation}
    \mathcal{P} = \left[ \begin{array}{c c}
    - \sqrt{\frac{\varepsilon r}{k}} & \sqrt{\frac{\varepsilon r}{k}}\\
    1 & 1
    \end{array} \right], \quad
    \mathcal{P}^{-1} = \frac{1}{2} \left[ \begin{array}{c c}
    -\sqrt{\frac{k}{\varepsilon r}} & 1 \\
    \sqrt{\frac{k}{\varepsilon r}} & 1
    \end{array} \right].
\end{equation}
The system is strictly hyperbolic with $2$ real and distinct eigenvalues.

\subsection{Augmented conservative hyperbolic heat equation}

According to the augmented scheme formulation in flux form in Sec. \ref{sectionflux}, the convective part of the system of equations in Eq. \ref{eq:hyp-heat}  can be rewritten in conservative form (see Eq. (\ref{eq:systemAugmented})), yielding
\begin{equation}\label{eq:hyp-heat-aug}
    \left\{
    \begin{aligned}
   &\frac{\partial u}{\partial t} 
   + r \frac{\partial q}{\partial x} = r\phi,\\
    &\frac{\partial q}{\partial t} + \frac{1}{\varepsilon}  
    \frac{\partial (ku)}{\partial x} = \frac{u}{\varepsilon}  
    \frac{\partial k}{\partial x}  - \frac{q}{\varepsilon}.
    \end{aligned}
    \right.
\end{equation}
This system can be written in matrix form as in Eq. \ref{eq:systemAugmented} with
\begin{equation}
    \bar{U} = \left[
    \begin{array}{c}
    u\\
    q\\
    k
    \end{array} \right], \;  \bar{F} = \left[ \begin{array}{c} 
    q \\
    \frac{k u }{\varepsilon} \\
    0
    \end{array} \right], \; \bar{\mathcal{K}} = \left[ \begin{array}{c c c} 
    0 & 0 & 0 \\
    0 & 0 & \frac{ u }{\varepsilon} \\
    0 & 0 & 0
    \end{array} \right], \; S = \left[ \begin{array}{c}
   r \phi\\
    - \frac{q}{\varepsilon}\\
    0
    \end{array} \right].
\end{equation}

 For the system in Eq. (\ref{eq:systemAugmented}), it is possible to define a Jacobian matrix of the flux $\bar{F}$ (see Eq. (\ref{eq:aug6}))
 \begin{equation}
  \bar{\mathcal{M}} = \left[ \begin{array}{c c c}
    0& r & 0\\
     \frac{k}{\varepsilon} &0 &  \frac{u}{\varepsilon} \\
     0 & 0 & 0
    \end{array} \right], 
\end{equation}

The eigenvalues of  the coefficient matrix $\bar{\mathcal{M}}$ read
\begin{equation}
    \lambda^1 = - \sqrt{\frac{k r}{\varepsilon}}, \quad  \lambda^2=0, \quad \lambda^3 = \sqrt{\frac{k r}{\varepsilon}},
\end{equation}
with corresponding right eigenvalue matrix $\mathcal{P}$, and its inverse as
\begin{equation}
    \bar{\mathcal{P}} = \left[ \begin{array}{c c c}
    - \sqrt{\frac{\varepsilon r}{k}} & {\frac{u}{k}} & \sqrt{\frac{\varepsilon r}{k}}\\
    1 & 0 & 1 \\
    0 & 1 & 0 \\
    \end{array} \right], \quad
    \bar{\mathcal{P}}^{-1} = \frac{1}{2} \left[ \begin{array}{c c c}
    -\sqrt{\frac{k}{\varepsilon r}} & 1 & -\frac{u}{ \sqrt{\varepsilon k r}} \\
    0 & 0 & 2\\
    \sqrt{\frac{k}{\varepsilon r}} & 1 &   \frac{u}{ \sqrt{\varepsilon k r}} 
    \end{array} \right].
\end{equation}
The system is strictly hyperbolic with $3$ real and distinct eigenvalues.

\end{document}